\documentclass[a4paper,11pt]{amsart}
\usepackage[hmarginratio={1:1},vmarginratio={1:1},lmargin=80.0pt,tmargin=90.0pt]{geometry}
\usepackage{mathtools}
\mathtoolsset{mathic}
\usepackage[final]{microtype}

\usepackage{latexsym,exscale,enumitem,amsfonts,amssymb,mathtools}
\usepackage{amsmath,amsthm,amsfonts,amssymb,amscd,stmaryrd,textcomp,xcolor,bbm}
\usepackage{nicefrac}
\usepackage{dynkin-diagrams}
\usepackage[normalem]{ulem}
\usepackage{todonotes}
\usepackage{thmtools}
\usepackage{etex}
\usepackage{hyperref}

\usepackage[all]{xy}
\usepackage{tikz}
\tikzset{anchorbase/.style={baseline={([yshift=-0.5ex]current bounding box.center)}},
  tinynodes/.style={font=\tiny,text height=0.75ex,text depth=0.15ex}
}
\usetikzlibrary{decorations.markings}
\usetikzlibrary{decorations.pathreplacing}
\usetikzlibrary{arrows,shapes,positioning}
\tikzstyle directed=[postaction={decorate,decoration={markings,
    mark=at position #1 with {\arrow{>}}}}]
\tikzstyle rdirected=[postaction={decorate,decoration={markings,
    mark=at position #1 with {\arrow{<}}}}]

\newtheorem{theorem}{Theorem}[section]
\newtheorem{lemma}[theorem]{Lemma}
\newtheorem{definition}[theorem]{Definition}

\newtheorem{remark}[theorem]{Remark}

\newtheorem{proposition}[theorem]{Proposition}

\newcommand{\fb}{\mathfrak{b}}

\newcommand{\fg}{\mathfrak{g}}
\newcommand{\fh}{\mathfrak{h}}

\newcommand{\re}{\mathrm{e}}
\newcommand{\rf}{\mathrm{f}}

\newcommand{\rI}{\mathrm{I}}

\newcommand{\cA}{\mathcal{A}}

\newcommand{\cF}{\mathcal{F}}
\newcommand{\cH}{\mathcal{H}}

\newcommand{\cO}{\mathcal{O}}

\newcommand{\cS}{\mathcal{S}}

\newcommand{\bC}{\mathbb{C}}

\newcommand{\bH}{\mathbb{H}}

\newcommand{\bQ}{\mathbb{Q}}

\newcommand{\bR}{\mathbb{R}}

\newcommand{\bZ}{\mathbb{Z}}

\newcommand{\OPB}{\hat{\mathrm{B}}}
\newcommand{\OPE}{\hat{\mathrm{E}}}
\newcommand{\OPF}{\hat{\mathrm{F}}}
\newcommand{\OPK}{\hat{\mathrm{K}}}
\newcommand{\OPL}{\hat{\mathrm{L}}}

\newcommand{\oneh}{\nicefrac{1}{2}}
\newcommand{\threeh}{\nicefrac{3}{2}}
\newcommand{\ellh}{\nicefrac{\ell}{2}}

\newcommand{\ellph}{\nicefrac{(\ell+1)}{2}}

\newcommand{\Uvmod}{U_v(\fg)\text{-}\mathrm{mod}}
\newcommand{\Uqmod}{U_{\mathrm{q}}\text{-}\mathrm{mod}}

\hypersetup{
  colorlinks = false,
  urlcolor = blue,
  pdfauthor = {XXX},
  pdfkeywords = {},
  pdftitle = {Combinatorial Fock spaces and quantum symmetric pairs},
  pdfsubject = {},
  pdfpagemode = UseNone,
  bookmarksopen = true,
  bookmarksopenlevel = 3,
  pdfdisplaydoctitle = true,
  linktocpage=true,
}

\begin{document}
\vbadness=10001
\hbadness=10001
\title[Combinatorial Fock spaces and quantum symmetric pairs]{Combinatorial Fock spaces and quantum symmetric pairs}

\author[M. Ehrig and K. Gan]{Michael Ehrig and Kaixuan Gan}

\address{M.E.: Beijing Institute of Technology, School of Mathematics and Statistics, Liangxiang Campus of Beijing Institute of Technology, Fangshan District, 100288 Beijing, China}
\email{michael.ehrig@bit.edu.cn}

\address{K.G.: Beijing Institute of Technology, School of Mathematics and Statistics, Liangxiang Campus of Beijing Institute of Technology, Fangshan District, 100288 Beijing, China}
\email{3120191408@bit.edu.cn}

\begin{abstract}
A way to construct the natural representation of the quantized affine algebra $U_v(\hat{\mathfrak{sl}}_\ell)$ is via the deformed Fock space by Misra and Miwa. This relates the classes of Weyl modules for $U_{\mathrm{q}}(\mathfrak{sl}_N)$ were ${\mathrm{q}}$ is a root of unity to the action of $U_v(\hat{\mathfrak{sl}}_\ell)$ as $N$ tends towards infinity. In this paper we investigate the situation outside of type $A$. In classical types, we construct embeddings of the Grothendieck group of finite dimensional $U_{\mathrm{q}}(\mathfrak{g})$-modules into Fock spaces of different charges and define an action of an affine quantum symmetric pair that plays the role of the quantized affine algebra. We describe how the action is related to the linkage principal for quantum groups at a root of unity and tensor product multiplicities.
\end{abstract}

\maketitle

\tableofcontents

%
%
\section{Introduction}

The Fock space (of charge zero) $\mathrm{F}_0$ arises in mathematical physics. In the context of representation theory it gives a particularly nice realisation of a representation for an affine Kac-Moody algebras (e.g. \cite[Chapter 14]{Kac}). With a basis labelled by all partitions, the basis elements can naturally be interpreted as the classes of irreducible finite dimensional highest weight modules $L(\lambda)$ for $\mathfrak{gl}_N(\bC)$ with dominant polynomial highest weights when $N$ tends towards infinity. One major step is to consider $\cF_0=\mathrm{F}_0\otimes_\bQ \bQ(v)$ the $\bQ(v)$-deformation of $\mathrm{F}_0$. Here and for the rest of the paper $v$ is an indeterminate. In this case Misra and Miwa \cite{MisraMiwa}, following the work of Hayahsi \cite{Hayashi} defined an action of the quantized universal enveloping algebra $U_v(\hat{\mathfrak{sl}}_\ell)$ on $\cF$. The action is of a particularly nice form in the sense that the image of a partition under a Chevalley generators has, as coefficients, monomials in $v$ that are combinatorially easy to describe.

In turn $\cF_0$ can be realized via the affine Hecke algebra, in this way it obtains the structure of a KL-module depending on a parameter $\ell$. This leads to a second natural basis, the Kazhdan-Lusztig basis. One now adopts the point of view that the standard basis of partitions corresponds to classes of Weyl modules $\Delta_{\mathrm{q}}(\lambda)$ for the quantum group $U_{\mathrm{q}}(\mathfrak{gl}_N)$ specialized at a $2\ell$'s root of unity and $N$ tending towards infinity. The Kazhdan-Lusztig basis corresponds to the classes of simple modules in the root of unity case and the transition matrix between the two basis evaluated at $1$ gives the multiplicities for the modules in terms of evaluations of parabolic affine Kazhdan-Lusztig modules, see \cite{LT}. From a Lie theoretic point of view the coefficients for the action of Chevalley generators can be connected to Shapovalov determinants as shown in \cite{RT}. Note that here and in the rest of the paper ${\mathrm{q}}$ is a root of unity such that ${\mathrm{q}}^2$ has order $\ell$.

The combinatorial Fock space was defined in \cite{LRS}. It is isomorphic, after extension of scalars, to the Grothendieck group of finite dimensional representations for $U_{\mathrm{q}}(\mathfrak{g})$, where $\fg$ is an arbitrary finite dimensional semi-simple complex Lie algebra. This comes naturally equipped with a basis corresponding to dominant weights or via the isomorphism to the classes of Weyl modules. Depending on the order $\ell$, \cite{LRS} establishes the structure of a KL-module on this space. This in turn yields a Kazhdan-Lusztig type basis that corresponds to the classes of simple modules with the transition matrix being given by parabolic affine Kazhdan-Lusztig polynomials. In type $A_N$ these Fock space models can be embedded in each other for growing $N$ and in the direct limit give the space $\cF_0$ from above.

This paper is a step to generalize the setup in the following sense. Outside of type $A$, we embed the combinatorial Fock space into a traditional Fock space, this plays the role of $\cF_0$. This is done via the combinatorics of certain sequences that are a natural generalisation of Young tableaux in type $A$. We then define an action of an affine quantum symmetric pair related to $U_v(\hat{\mathfrak{sl}}_\ell)$ on this Fock space. Such algebras were studied and classified in the non-affine situation (see \cite{Letzter1} and \cite{Letzter2}) and in the Kac-Moody setting (see \cite{Kolb}). This action is compatible with the linkage principal for the quantum group at a root of unity and describes certain tensor product multiplicities for Weyl modules in a similar way as the action of the quantum affine algebra does in type $A$. Roughly said the action of generators describes, in the Grothendieck group, the image of a Weyl module under the translation functor given by taking the tensor product with the ``natural'' representation for a fixed rank of $\fg$.

The structure of the paper is as follows. In Section \ref{sec:prelim} we recall the necessary definitions for weight combinatorics for quantum groups, the definition of the combinatorial Fock space from \cite{LRS}, and some facts about quantum groups at a root of unity and their representation theory. In Section \ref{sec:fockspace} we introduce the combinatorial definition of the analogue of $\cF_0$ in the type $A$ setting as a space of certain sequences as well as most of the combinatorics for sequences that are needed later on. We briefly recall the type $A$ situation translated into these combinatorics. We introduce some notions for affine Weyl groups and alcove combinatorics needed later on, as well as the affine quantum symmetric pair that acts outside of type $A$. In Section \ref{sec:typeC} we investigate the type $C$ case. This is the easiest case with the least amount of complications. We first specialize the combinatorics to this case, then relate linear operators on the Fock space to the linkage principal in type $C$ and then define the action of the affine quantum symmetric pair (Definition \ref{def:quantumsymoperatorsC}). In Theorem \ref{thm:actionanddecompC} we describe the relationship between the action and tensor product multiplicities and in Proposition \ref{prop:operatorsasdecomp} the analogue of letting $N$ tend to infinity here. Section \ref{sec:typeB} has then the same structure but for type $B$. Note that here one has to distinguish between the case of $\ell$ even and $\ell$ odd. The definition of the action can be found in Definitions \ref{def:quantumsymoperatorsB} and \ref{def:quantumsymoperatorsB2}. The relationship with tensor product multiplicities is divided into Theorems \ref{thm:actionanddecompB1}, \ref{thm:actionanddecompB2}, \ref{thm:actionanddecompB3}, and \ref{thm:actionanddecompB4}. The situation of letting $N$ tend towards infinity is described in and before Proposition \ref{prop:operatorsasdecompBhint} for $\ell$ odd. The same results for $\ell$ even can be derived in a similar way. Finally in Section \ref{sec:typeDbeyond} we give a sketch of how to apply this construction to type $D$ and what the corresponding results look like. We also elaborate on why the construction is not fully detailed in the paper.

\subsection*{Acknowledgements}\label{subsec:acknowledgements}

We thank Daniel Tubbenhauer for comments and remarks. This work was supported by the National Natural Science Foundation of China under Grant No. 12050410261.



\section{Preliminaries} \label{sec:prelim}

In this part we review the fundamental definitions for the three main objects involved. The weight combinatorics of semi-simple Lie algebras, the combinatorial Fock space, and the quantum group at a root of unity corresponding to a non-exceptional semi-simple complex Lie algebra. Throughout the paper we use an integer $\ell$ that is the order of a root of unity in Section \ref{sec:quantumgroup}. In general we assume
\[
\ell > 3.
\]
This is done to avoid special cases of quantum symmetric pairs in case of small ${\ell}$.

\subsection{Lie algebras and weight combinatorics}

We fix $\fg$ a finite dimensional semi-simple complex Lie algebra. In the following we are only interested in the non-exceptional cases. Hence from now on we make the assumption that $\mathfrak{g}$ is of type $A_N$, $B_N$, $C_N$ or $D_N$. We denote this type by $X_N$ in what follows. To not having to deal with the various isomorphisms in small ranks we assume $N>1$ in type $B_N$, $N>2$ in type $C_N$, and $N>3$ in type $D_N$.

\begin{remark}
All constructions outside of Sections \ref{sec:typeC} and \ref{sec:typeB} can be made for exceptional Lie algebras. We go into more details why we are restricting to the classical cases in Section \ref{sec:typeDbeyond}.
\end{remark}

Fix a Cartan subalgebra $\fh$ and a Borel subalgebra $\fb$ containing $\fh$. With the choice of $\fh$ we denote by $\Phi$ the root system of $\fg$ and by $\Phi^+$ the positive roots with respect to $\fb$. By $W$ we denote the Weyl group corresponding to $\fg$. Fix a non-degenerate $W$-invariant bilinear form $(-,-):\fh^\ast \times \fh^\ast \rightarrow \bC$. For $\alpha \in \Phi^+$, the corresponding coroot is defined by $\alpha^\vee = 2\alpha / (\alpha,\alpha)$. We label the simple roots in $\Phi^+$ by $\alpha_1,\ldots,\alpha_N$. The lattice of integral weights is denoted by
\[
X = \{ \lambda \in \fh^\ast \mid (\lambda,\alpha^\vee) \in \bZ, \text{ for } \alpha \in \Phi^+ \}.
\]
The elements $\omega_i\in X$ such that $(\omega_i,\alpha_j^\vee) = \delta_{ij}$ are called the fundamental weights. Inside the weight lattice we fix the set of dominant weights
\[
X^+ = \{ \lambda \in X \mid (\lambda,\alpha^\vee) \geq 0, \text{ for } \alpha \in \Phi^+ \}.
\]
On $X$ we have the action of the Weyl group $W$ with the reflection $s_\alpha$ given by $s_{\alpha}(\lambda) = \lambda - (\lambda,\alpha^\vee)\alpha$ for $\alpha \in \Phi^+$.

\subsection{The combinatorial Fock space}

Fix the element $\rho = \frac{1}{2} \sum_{\alpha \in \Phi^+} \alpha$ and the set of $\rho$-shifted dominant weights $X^+_\rho=X^+ +\rho$. We denote by $\bQ(v)$ the rational functions over $\bQ$. 

Following \cite[Section 1.1]{LRS}, the \emph{combinatorial Fock} space $\cF(X_N)$ of type $X_N$ is the $\bQ(v)$-vector space with basis $\{\underline{\lambda}=\lambda+\rho \mid \lambda \in X^+\}$. 

\begin{remark}
In \cite{LRS} the Fock space is originally defined to be generated as a $\bZ[v,v^{-1}]$-module by elements indexed by $X$ modulo relations, but \cite[Theorem 1.1]{LRS} shows that the elements indexed by $X^+$ form a basis as a free $\bZ[v,v^{-1}]$-module. We simple shift the indexing set of the basis elements by $\rho$ and extend scalars to the field $\bQ(v)$.
\end{remark}

We do not make use of the description, but it should be noted that one of the main and most intricate results of \cite{LRS} is to endow the combinatorial Fock space $\cF(X_N)$ with the structure of a KL-module. For this it is alternatively constructed by starting from the affine Hecke algebra. For this purpose an action of an affine Weyl group has to be introduced and this action depends on an integer $\ell$. Thus in contrast to \cite{LRS} we drop the label $\ell$ in the notation for the Fock space, as we do not make use of the description via affine Hecke algebras.

\subsection{Quantum group at a root of unity} \label{sec:quantumgroup}

Starting with a fixed semi-simple complex Lie algebra of type $X_N$ we define the corresponding quantized enveloping algebra, following \cite{LusztigQuantum}. For $1 \leq i,j \leq N$ we set $a_{ij} = (\alpha_j,\alpha_i^\vee) \in \bZ$. For $1 \leq i \leq N$ we fix $d_i = 1$ if $\alpha_i$ is short root and $d_i = 2$ if $\alpha_i$ is a long root. Let $v_i = v^{d_i} \in \bQ(v)$. This is extended to all roots, by setting $d_\alpha = d_i$ if $\alpha$ and $\alpha_i$ are in the same Weyl group orbit.

\begin{remark}
Note that in type $A_N$ and type $D_N$ all simple roots are short. For the non simply-laced cases we have the Dynkin diagrams
\[
B_N: \quad \dynkin{B}{} \quad \text{ and }\quad C_N: \quad \dynkin{C}{}.
\]
Hence in type $B_N$ there is precisely one short simple root $\alpha_N$ corresponding to the right most vertex in the diagram above, while in type $C_N$ there is precisely one long simple root, again corresponding to the right-most vertex in the diagram above. For $(d_1,\ldots,d_N)$, with $d_i$ corresponding to the $i$-th node from the left in the diagrams above, we have in type $B_N$ the vector $(2,2,\ldots,2,1)$ and in type $C_N$ the vector $(1,1,\ldots,1,2)$.
\end{remark}

To fix the notation for the quantum integers we define for $n \in \bZ_{\geq 0}$ and $k \in \bZ$ the following elements in $\bZ[v,v^{-1}]$
\[
[n]_v = \frac{v^n-v^{-n}}{v-v^{-1}},\,[n]_v! = \prod_{m=1}^n \frac{v^m-v^{-m}}{v-v^{-1}}, \text{ and } 
\left[ 
\begin{array}{c}
k \\ n
\end{array}
\right]_v = \prod_{m=1}^n \frac{v^{k+1-m} - v^{-k-1+m}}{v^m-v^{-m}}.
\]
We use the same notation if we substitute $v$ for $v_i$ or an element of a field.

The \emph{quantized enveloping algebra} $U_v(\fg)$ is the associative algebra over $\bQ(v)$ generated by elements $\{E_i,F_i,K_i^{\pm 1} \mid 1\leq i \leq N\}$ subject to the following relations for all $1\leq i,j \leq N$
\begin{enumerate}
\item $K_iK_i^{-1} = 1 = K_i^{-1}K_i$, $K_iK_j = K_j K_i$,
\item $K_iE_j = v_i^{a_{ij}}E_jK_i$, $K_iF_j = v_i^{-a_{ij}}F_jK_i$,
\item $E_iF_j - F_jE_i = \delta_{ij} \frac{K_i-K_i^{-1}}{v_i-v_i^{-1}}$ (commutator relation),
\item $\sum_{k=0}^{1-a_{ij}} (-1)^k E_i^{(1-a_{ij}-k)}E_jE_i^{(k)} = 0$, for $i \neq j$ (quantum Serre relations),
\item $\sum_{k=0}^{1-a_{ij}} (-1)^k F_i^{(1-a_{ij}-k)}F_jF_i^{(k)} = 0$, for $i \neq j$ (quantum Serre relations),
\end{enumerate}
where $E_i^{(k)} = E_i^k/[k]_{v_i}!$ and $F_i^{(k)} = F_i^k/[k]_{v_i}!$ are the divided powers.

For a $U_v(\fg)$-module $M$ and $\lambda \in X$ we fix the $\lambda$-weight space as
\[
M_\lambda = \left\lbrace m\in M \mid K_i m = v_i^{(\lambda,\alpha_i^\vee)}m\right\rbrace.
\]
We denote by $\Uvmod$ the full subcategory of finite dimensional $U_v(\fg)$-modules consisting of modules $M$ of type $1$, i.e. those finite dimensional modules satisfying $M=\bigoplus_{\lambda \in X} M_\lambda$. We denote by $[\Uvmod]$ the Grothendieck group of $\Uvmod$ and to avoid unnecessary clutter of notations, we assume that the scalars of the Grothendieck group are extended from $\bZ$ to $\bQ(v)$.

Following \cite{LusztigRootofUnity}, we fix the subring $A=\bZ[v,v^{-1}] \subset \bQ(v)$ and denote by $U_A(\fg)$ the $A$-form of $U_v(\fg)$, i.e. the $A$ subalgebra of $U_v(\fg)$ generated by divided powers $E_i^{(k)}$, $F_i^{(k)}$, $K_i^{\pm 1}$ and the elements
\[
\left[ 
\begin{array}{c}
K_i;c \\ k
\end{array}
\right]=\prod_{s=1}^k \frac{K_i v_i^{c+1-s}-K_i^{-1} v_i^{s-1-c}}{v_i^{s}-v_i^{-s}} \text{ for }c\in\bZ,k\in\bZ_{>0}.
\]
Note that for $k=1$ and $c=0$ this is exactly equal to $[E_i,F_i]$, while other elements of this form appear in generalised versions of commutator relation for divided powers.

Fix ${\mathrm{q}}\in\bC$ with ${\mathrm{q}}^{2}$ is primitive $\ell$-th root of unity and consider $U_{\mathrm{q}} = U_A(\fg) \otimes_A \bC$, where $v\in A$ acts by multiplication with ${\mathrm{q}}$ in $\bC$. For a $U_{\mathrm{q}}$-module $M$ and $\lambda \in X$ we set
\[
M_\lambda = \left\lbrace m\in M \mid K_i m = {\mathrm{q}}^{d_i(\lambda,\alpha_i^\vee)}m,\, \left[ 
\begin{array}{c}
K_i;0 \\ k
\end{array}
\right]m=
\left[ 
\begin{array}{c}
(\lambda,\alpha_i^\vee) \\ k
\end{array}
\right]_{{\mathrm{q}}^{d_i}}m \right\rbrace
\]
We refer to $U_{\mathrm{q}}$ simply as the \emph{quantum group at a root of unity}. 

In analogy to the generic case, we denote by $\Uqmod$ the full subcategory of finite dimensional $U_{\mathrm{q}}$-modules consisting of modules $M$ such that $M=\bigoplus_{\lambda \in X} M_\lambda$. We denote again by $[\Uqmod]$ the Grothendieck group and assume that the scalars are extended from $\bZ$ to $\bQ(v)$.

In both cases of $\Uvmod$ and $\Uqmod$ we have the character of a module given as $\mathrm{ch}(M) = \sum_{\lambda \in X} \mathrm{dim}(M_\lambda) \mathbf{e}^\lambda$ for formal symbols $\mathbf{e}^\lambda$ and the dimension taken over the respective fields $\bQ(v)$ and $\bC$. Replacing the dimension by the rank over $A$, one obtains the character of a $U_A(\fg)$-module that is free as an $A$-module.

In the generic case, for $\lambda\in X^+$, we denote by $\Delta_v(\lambda)=L_v(\lambda)$ the irreducible highest weight module of highest weight $\lambda$ and we fix $x_\lambda \in L_v(\lambda)$ a highest weight vector. Following \cite[Section 7]{Tanisaki}, this can be lifted to the $A$-form by defining $\Delta_A(\lambda)$ to be the $U_A(\fg)$-submodule of $\Delta_v(\lambda)$ generated by $x_\lambda$. Alternatively, \cite{APW} construct the module as a quotient of the Verma module defined for $U_A(\fg)$. Then $\Delta_A(\lambda)$ is a free $A$-module with $\Delta_v(\lambda)=\Delta_A(\lambda) \otimes_A \bQ(v)$. Especially one obtains, see \cite[Proposition 1.22]{APW}, 
\begin{align*}
\mathrm{ch}(\Delta_v(\lambda)) = \mathrm{ch}(\Delta_A(\lambda)),
\end{align*}
both of them given by Weyl's character formula. As noted in \cite[Remark 1.25]{APW} this does not follow the usual tradition from algebraic groups to construct Weyl modules. In those setting one would induce from the opposite Borel to obtain the dual Weyl module and then use the duality to obtain the Weyl module. This more involved approach has many advantages, but since we are only interested in the classes of Weyl modules in the Grothendieck group we use this simpler construction here.

Since $\Delta_A(\lambda)$ is free over $A$, we can set $\Delta_{\mathrm{q}}(\lambda)=\Delta_A(\lambda) \otimes_A \bC$ and naturally
\begin{align} \label{eq:equalcharacters2}
\mathrm{ch}(\Delta_{\mathrm{q}}(\lambda)) = \mathrm{ch}(\Delta_A(\lambda)) = \mathrm{ch}(\Delta_v(\lambda)).
\end{align}
The module $\Delta_{\mathrm{q}}(\lambda)$ is called the \emph{Weyl module} with highest weight $\lambda$. By construction it is a quotient of the corresponding Verma module $M_{\mathrm{q}}(\lambda)$ of highest weight $\lambda$ and $\Delta_{\mathrm{q}}(\lambda)$ has a unique irreducible quotient $L_{\mathrm{q}}(\lambda)$. Note that every simple object in $\Uqmod$ can be obtained in this way.

\begin{remark} \label{rem:weylarebasis}
Since $\Delta_{\mathrm{q}}(\lambda)$ has $L_{\mathrm{q}}(\lambda)$ as its head and by highest weight theory all other composition factors are of the form $L_{\mathrm{q}}(\mu)$ for $\lambda - \mu$ a positive, non-zero, sum of positive roots, the classes $[\Delta_{\mathrm{q}}(\lambda)]$ for $\lambda \in X^+$ form a basis of the Grothendieck group $[\Uqmod]$.
\end{remark}

The Weyl modules play the key role for our situation here. The following observation is well-known to experts, but since it is central to our arguments we formulate it here.

\begin{proposition} \label{prop:tensorproductdecomp}
Let $\lambda,\mu \in X^+$ and consider the tensor product decomposition
\[
\Delta_v(\lambda) \otimes_{\bQ(v)} \Delta_v(\mu) \cong \bigoplus_{i=1}^r \Delta_v(\nu_i)
\]
in $\Uvmod$ for some dominant weights $\nu_1,\ldots,\nu_r$. Then in the Grothendieck group $[\Uqmod]$ it holds
\[
[\Delta_{\mathrm{q}}(\lambda) \otimes_{\bC} \Delta_{\mathrm{q}}(\mu)] = \sum_{i=1}^r [\Delta_{\mathrm{q}}(\nu_i)].
\]
\end{proposition}
\begin{proof}
Consider the tensor product $\Delta_A(\lambda) \otimes_A \Delta_A(\mu)$. Then, by \eqref{eq:equalcharacters2}, it holds 
\[
\mathrm{ch}(\Delta_{\mathrm{q}}(\lambda) \otimes \Delta_{\mathrm{q}}(\mu)) = \mathrm{ch}(\Delta_A(\lambda) \otimes_A \Delta_A(\mu)) = \mathrm{ch}(\Delta_v(\lambda) \otimes_{\bQ(v)} \Delta_A(\mu))
\]
Since the classes of Weyl modules form a basis of $[\Uqmod]$ and \eqref{eq:equalcharacters2}, we thus obtain that the decomposition of the character in the generic case, which is just the tensor product decomposition, gives the decomposition in the Grothendieck group in the root of unity case.
\end{proof}

We only need the statement about the Grothendieck group from Proposition \ref{prop:tensorproductdecomp}, but a stronger statement in the category also holds.

\begin{remark}
With the same notations as in Proposition \ref{prop:tensorproductdecomp}, $\Delta_{\mathrm{q}}(\lambda) \otimes_{\bC} \Delta_{\mathrm{q}}(\mu)$ has a filtration with subquotient isomorphic to Weyl modules of the form $\Delta_{\mathrm{q}}(\nu_i)$ for $\nu_1,\ldots,\nu_r$ some dominant weights.

The existence of a filtration with Weyl modules as subquotient can be derived from the literature in different ways. In \cite{AST} (see arXiv-Appendix, Claim 3.10.1) this is worked out in the simply-laced case and for dual Weyl modules (hence one needs to apply a duality). In \cite[Theorem 3.3]{Paradowski} this can be found more general under the name of good filtrations, which then dualizes to a filtration by Weyl modules. To determine which Weyl modules appear in the filtration one uses the same argument about characters as in Proposition \ref{prop:tensorproductdecomp} above.
\end{remark}

In contrast to $\Uvmod$ which is a semi-simple category, $\Uqmod$ is not semi-simple. For a criterion to decide when two simple modules respectively two Weyl modules can be in the same block, one introduces an action of an affine reflection group on $X$. Namely for $\alpha \in \Phi^+$ define $\ell_\alpha = \ell / \mathrm{gcd}(\ell,d_\alpha)$. Then denote by $W_\ell$ the group generated by the affine reflections of the form
\[
s_{\alpha,k}\cdot \lambda = s_\alpha \cdot \lambda + k \ell_\alpha \alpha,
\]
for $\alpha \in \Phi^+$, $k\in\bZ$ and $w \cdot \lambda = w(\lambda +\rho) - \rho$ for $w\in W$ the \textit{dot} action of $W$ on $X$.

\begin{remark}
Note that in our situation $d_\alpha$ is either $1$ or $2$, hence in case that $\ell$ is odd $\ell_\alpha=\ell$ for all $\alpha$. In which case $W_\ell$ is the affine Weyl group attached to $W$, except that the action is scaled by a factor $\ell$ after the shift by $\rho$.
In case that $\ell$ is even $\ell_\alpha = \ell / 2$ for a long root $\alpha$. In this case the group $W_\ell$ is acting as the affine Weyl group for the dual root system, shifted by $\rho$ and scaled by a factor $\ell$.
\end{remark}

We call two weights $\lambda,\mu\in X^+$ linked, if there exists an element $w\in W_\ell$ such that $\lambda = w\cdot \mu$. This correlates to extensions between simple modules as follows.

\begin{theorem}{\cite[Theorem 4.3]{Andersen}}
Let $\lambda,\mu \in X^+$. If $\mathrm{Ext}^1(L_{\mathrm{q}}(\lambda),L_{\mathrm{q}}(\mu))\neq 0$ then $\lambda$ and $\mu$ are linked and not equal.
\end{theorem}

Since $\Delta_{\mathrm{q}}(\lambda)$ for $\lambda\in X^+$ is indecomposable, we thus get that if $\Delta_{\mathrm{q}}(\lambda)$ and $\Delta_{\mathrm{q}}(\mu)$ are in the same block, then $\lambda$ and $\mu$ are linked.

As mentioned, the $L_{\mathrm{q}}(\lambda)$ for $\lambda \in X^+$ form a complete set of irreducible modules in $\Uqmod$. We fix the following identification from now on to view classes of Weyl modules as basis elements of the corresponding combinatorial Fock space.

\begin{lemma}
There is a $\bQ(v)$-vector space isomorphism between $\cF(X_N)$ and $[\Uqmod]\otimes_\bZ \bQ(v)$, mapping the basis vector $\underline{\lambda}=\lambda+\rho$ to $[\Delta_{\mathrm{q}}(\lambda)]$.
\end{lemma}

Note that the interpretation of $\cF(X_N)$ as a KL-module in \cite{LRS} allows to identify the KL-basis of $\cF(X_N)$ with the basis given by the classes of irreducible modules in $[\Uqmod]\otimes_\bZ \bQ(v)$.

\section{Fock spaces, affine Weyl groups, and quantum symmetric pairs} \label{sec:fockspace}
In this section we introduce the necessary combinatorics of sequences to define the spaces that the $\cF(X_N)$ are embedded into.

Denote by $\bH=\oneh+\bZ$ the half-integers. Since $\bZ$ acts on $\bH$ by addition, we consider the cosets $\bH/r\bZ$ for a positive integer $r$. For $p\in\bH$ we denote by $\overline{p}$ its coset in $\bH/r\bZ$ and for $p\in\bZ$, we denote by $\overline{p}$ its coset in $\bZ/r\bZ$.

Consider the following sets of sequences
\begin{align*}
\mathcal{S}_{\bZ} &= \{ \underline{a}:\bZ \rightarrow \{0,1\} \mid \underline{a}(i) = 0 \text{ for } i \gg 0,\, \underline{a}(i) = 1 \text{ for } i \ll 0\} \text{ and}\\
\mathcal{S}_{\bH} &= \{ \underline{a}:\bH \rightarrow \{0,1\} \mid \underline{a}(i) = 0 \text{ for } i \gg 0,\, \underline{a}(i) = 1 \text{ for } i \ll 0\}.
\end{align*}
We call $\mathcal{S}_{\bZ}$ the sequences supported on integers and $\mathcal{S}_{\bH}$ the sequences supported on half-integers. Then we denote by $\cF=\cF^1$ the $\bQ(v)$-vector space with basis $\mathcal{S}_{\bZ}$ and by $\cF^{\oneh}$ the $\bQ(v)$-vector space with basis $\mathcal{S}_{\bH}$. We refer to $\cF$ and $\cF^{\oneh}$ simply as the \emph{Fock space}. The notation $\cF^1$ is needed for type $B$ to make the differentiation between Fock spaces clear in that case.

\begin{remark}
While they are defined as maps we consider these as $\{0,1\}$-sequences with indices labelled by either $\bZ$ or $\bH$. We refer to the value $\underline{a}(i)$ as the entry at position $i$ and say that the position is ``empty'' if $\underline{a}(i)=0$ and it is occupied if $\underline{a}(i)=1$.
\end{remark}

Since a sequence in this set only has finitely many non-zero entry in its positive half and only finitely many zero entry in its non-positive half, the following is well defined for $\underline{a} \in \mathcal{S}_{\bZ}$ and equally for $\underline{a} \in \mathcal{S}_{\bH}$
\[
\mathrm{ch}(\underline{a}) = \sum_{i> 0} \underline{a}(i) - \sum_{i \leq 0} (1-\underline{a}(i)).
\]
We call this the charge of $\underline{a}$ and the set of all sequences of charge $N$ is denoted by $\mathcal{S}_{\bZ,N}$ respectively $\mathcal{S}_{\bH,N}$. The corresponding subspaces, the Fock space of charge $N$, are spanned by sequences of charge $N$ and denoted by $\cF_N=\cF_N^1$ and $\cF_N^{\oneh}$.

\subsection{Moving operators and counting statistics} \label{sec:operatorsandstats}

We define two basic operators on Fock spaces that are moving a $1$ entry to either the left or right.

\begin{definition} \label{def:movingoperators}
For $\underline{a} \in \mathcal{S}_{\bZ}$ let $i\in\bH$ and for $\underline{a} \in \mathcal{S}_{\bH}$ let $i\in\bZ$.
\begin{itemize}
\item Define a sequence $\underline{b}$ via $\underline{b}(i-\oneh)-\underline{b}(i+\oneh)=1$ and $\underline{b}(j) = \underline{a}(j)$ for $j \neq i \pm \oneh$. Then
\[
\re_i \underline{a} = \begin{cases}
\underline{b} & \text{ if } \underline{a}(i+\oneh) - \underline{a}(i-\oneh) =1\\
0 & \text{ otherwise.}
\end{cases}
\]
\item Define a sequence $\underline{c}$ via $\underline{c}(i+\oneh)-\underline{c}(i-\oneh)=1$ and $\underline{c}(j) = \underline{a}(j)$ for $j \neq i \pm \oneh$. Then
\[
\rf_i \underline{a} = \begin{cases}
\underline{c} & \text{ if } \underline{a}(i-\oneh) - \underline{a}(i+\oneh) =1\\
0 & \text{ otherwise.}
\end{cases}
\]
\end{itemize}
\end{definition}

These define linear operators on both $\cF$ and on $\cF^{\oneh}$, which we call the \emph{moving operators}. By definition $\re_i$ and $\rf_i$ preserve the charge and thus restrict to linear operators on $\cF_N$ and $\cF^{\oneh}_N$. We say that $\re_i$ moves an entry $1$ from position $i+\oneh$ to position $i-\oneh$ or is zero if that is not possible, while $\rf_i$ moves an entry $1$ in the opposite direction. 

For the definition of the action of the quantum affine algebra in type $A$ or the quantum symmetric pair in other types, we need a number of counting statistics, which we introduce now. They appear in different combination for all the actions.

\begin{definition} \label{def:countingstats}
Fix $r\in\bZ$. For $\underline{a} \in \mathcal{S}_{\bZ}$ let $j\in\bH$ and for $\underline{a} \in \mathcal{S}_{\bZ}$ let $j\in\bZ$. Then define
\begin{align*}
R_r^e(j,\underline{a}) = \# \{ k \in j + r\bZ_{>0} \mid \re_k \underline{a} \neq 0\}, & \quad R_r^f(j,\underline{a}) = \# \{ k \in j + r\bZ_{>0} \mid \rf_k \underline{a} \neq 0\}, \\
L_r^e(j,\underline{a}) = \# \{ k \in j - r\bZ_{>0} \mid \re_k \underline{a} \neq 0\}, & \quad L_r^f(j,\underline{a}) = \# \{ k \in j - r\bZ_{>0} \mid \rf_k \underline{a} \neq 0\}.
\end{align*}
To shorten notation we also introduce
\begin{align*}
R_r^{e-f}(j,\underline{a}) = R_r^e(j,\underline{a}) - R_r^f(j,\underline{a}), & \quad  R_r^{f-e}(j,\underline{a}) = R_r^f(j,\underline{a}) - R_r^e(j,\underline{a}),\\
L_r^{e-f}(j,\underline{a}) = L_r^e(j,\underline{a}) - L_r^f(j,\underline{a}), & \quad  L_r^{f-e}(j,\underline{a}) = L_r^f(j,\underline{a}) - L_r^e(j,\underline{a}).
\end{align*}
In addition, for $\underline{a} \in \mathcal{S}_{\bZ}$ let $\overline{i} \in \bH/r\bZ$ and for $\underline{a} \in \mathcal{S}_{\bZ}$ let $\overline{i} \in \bZ/r\bZ$. Then define
\begin{align*}
T_r^e(\overline{i},\underline{a}) = \# \{ j \in \overline{i} \mid \re_j \underline{a} \neq 0\} \text{ and }, & \quad T_r^f(\overline{i},\underline{a}) = \# \{ j \in \overline{i} \mid \rf_j \underline{a} \neq 0\},\\
T_r^{e-f}(\overline{i},\underline{a}) = T_r^e(\overline{i},\underline{a})-T_r^f(\overline{i},\underline{a}), & \quad T_r^{f-e}(\overline{i},\underline{a})=T_r^f(\overline{i}-\underline{a})-T_r^e(\overline{i},\underline{a}).
\end{align*}
\end{definition}

The cumbersome notation is necessary, since the action of the generators in the different cases depend on both the sequence $\underline{a}$ as well as a position $j$ such that an operator $\re_j$ respectively $\rf_j$ can be applied. In most cases the index $r$ is $r=\ell$, but in case of type $B$ and even $\ell$ it has to replaced by $r=\ellh$ for the action, hence the addition of the index $r$ in the definition.

\begin{remark}
Fix a sequence $\underline{a}$ in either $\mathcal{S}_{\bZ}$ or $\mathcal{S}_{\bH}$. Then $R_r^e(j,\underline{a})$ counts how often, for $k \in j + r\bZ_{>0}$ (i.e. to the right of $j$), it happens that $\underline{a}(k-\oneh)=0$ and $\underline{a}(k+\oneh)=1$. Similarly $R_r^f(j,\underline{a})$ counts when $\underline{a}(k-\oneh)=1$ and $\underline{a}(k+\oneh)=0$ instead and $L_r^e(j,\underline{a})$ and $L_r^f(j,\underline{a})$ count these occurrences for $k \in j - r\bZ_{>0}$, i.e. to the left of $j$. While the statistics of the form $T_r^e(\overline{i},\underline{a})$ count where such situation occur on all of $\bZ$ respectively $\bH$. This should give an easy way to remember which of the operators counts what and in which direction of a fixed position.
\end{remark}

\subsection{Type A} \label{sec:typeA}

We briefly recall here the definition of the action on the Fock space in type $A$. Instead of defining the Fock space via partitions as usual, we use sequences.

Fix $X=\oplus_{i=1}^N \bZ \varepsilon_i$ the a lattice of integral weights for $\mathfrak{gl}_N(\bC)$. One could also use $\mathfrak{sl}_N(\bC)$ here, but the general linear case is more convenient from the view point of combinatorics. In $X$ we have
\begin{align*}
X^+ &= \{(\lambda_1,\ldots,\lambda_N) \in X \mid \lambda_1 \geq \lambda_2 \geq \ldots \geq \lambda_N\} \text{ and}\\
P^+ &= \{\lambda_1,\ldots,\lambda_N) \in X^+ \mid \lambda_N \geq 0\},
\end{align*}
the sets of dominant weights and polynomial weights of $\mathfrak{gl}_N(\bC)$. One should consider $P^+$ as the analogue of the dominant weights for $\mathfrak{sl}_N(\bC)$. Furthermore, fix the element $\rho = (0,-1,\ldots,-(N-1))$. In contrast to $\mathfrak{sl}_N(\bC)$ we have a choice here and we fix this particular $\rho$ to match up nicely with the usual definition of partitions and adding boxes of certain residues.

To $\lambda \in P^+$ we associate the sequence $\underline{a}_{\underline{\lambda}} \in \cS_{\bZ}$ given by 
\[
\underline{a}_{\underline{\lambda}}(i) = \begin{cases}
1 & \text{if } i \leq -N, \\
1 & \text{if there exists } 1 \leq j \leq N \text{ s.t. } i=\underline{\lambda}_j, \\
0 & \text{otherwise.}
\end{cases}
\]
i.e. we put a $1$ at all positions that appear as entries in $\underline{\lambda}$, since we added $\underline{\lambda}=\lambda+\rho$, these are all distinct, so this is well-defined. In addition we set all values to $1$ for $i \leq -N$. Note that for the zero weight $\underline{0} = (0,-1,\ldots,-(N-1))$, hence the sequence associated to the zero weight has value $1$ for all non-positive integers and value $0$ for all positive values.

Note that this embeds $\cF(A_N)$ into the Fock space $\cF_0$ of charge zero.

\begin{remark} \label{rem:sequencestopartitions}
To pass to the partition description of $\cF_0$, see for example \cite{Ariki} or \cite{RT}, fix a sequence $\underline{a}$ of charge $0$. Let $(\lambda_1,\lambda_2,\ldots)$ be defined such that $\underline{a}(\lambda_1)$ is the right most entry equal to $1$, $\underline{a}(\lambda_2-1)$ is the second right most entry equal to $1$ and so forth with $\underline{a}(\lambda_i - (i-1))$ being the $i$-th right most entry equal to $1$. Then by construction $(\lambda_1,\lambda_2,\ldots)$ is a weakly decreasing sequence of non-negative integers that eventually stabilizes to $0$ after finitely many steps, hence a partition.
\end{remark}

For $\overline{p} \in \bH/\ell\bZ$ and $\underline{a}\in\mathcal{S}_{\bZ}$, define linear operators on $\cF$ via
\begin{align*}
\OPE_{\overline{p}}\, \underline{a} = \sum_{j \in \overline{p}} v^{R_\ell^{e-f}(j,\underline{a})} \re_j\,\underline{a}, \quad \OPF_{\overline{p}}\, \underline{a} = \sum_{j \in \overline{p}} v^{L_\ell^{f-e}(j,\underline{a})} \rf_j\,\underline{a}, \text{ and } \OPK_{\overline{p}}\, \underline{a} &= v^{T_\ell^{f-e}(\overline{p},\underline{a})} \underline{a}.
\end{align*}
Restricting these operators to the subspace $\cF_0$ gives then the well-known action on $\cF_0$, see for example \cite{Ariki}, by using Remark \ref{rem:sequencestopartitions} to translate it to sequences instead of partitions. 

\begin{theorem}{\cite[Theorem 10.6]{Ariki}}
The linear operators $\{\OPE_{\overline{p}},\OPF_{\overline{p}},\OPK_{\overline{p}}^{\pm 1} \}$ define an action of the quantum affine algebra $U_v(\hat{\mathfrak{sl}}_\ell)^\prime $ on $\cF_0$.
\end{theorem}

Note that there is a natural isomorphism between $\Psi_m:\cF_0\rightarrow \cF_m$ for any integer $m$, by just shifting the sequence by $m$ steps, i.e. $\Psi_m(\underline{a})(i)=\underline{a}(i-m)$. By construction
\[
\Psi_m(\OPE_{\overline{p}}\underline{a}) = \OPE_{\overline{p+m}}\Psi_m(\underline{a}),\, \Psi_m(\OPF_{\overline{p}}\underline{a}) = \OPF_{\overline{p+m}}\Psi_m(\underline{a}), \text{ and } \Psi_m(\OPK_{\overline{p}}\underline{a}) = \OPK_{\overline{p+m}}\Psi_m(\underline{a}).
\]
Since all relations of the quantum affine algebra are rotation invariant, this defines an action of $U_v(\hat{\mathfrak{sl}}_\ell)^\prime$ on all $\cF_m$, $m\in\bZ$, given by the same formulas. The only difference really being for which index $\overline{p}$, $\OPF_{\overline{p}}$ does not act as zero on the ``highest weight sequence'' where all $1$'s are as far to the left as possible, i.e. the sequence on which all $\OPE_{\overline{p}^\prime}$ act as zero.

\subsection{Affine Weyl group combinatorics}

To introduce alcove geometry for the affine Weyl group we define $X_\bR=X \otimes_\bZ \bR$. Then for $\alpha \in \Phi^+$ we consider the affine hyperplane
\[
H_{\alpha,k}=\{x \in X_\bR \mid (\lambda + \rho,\alpha^\vee)=k\ell_\alpha\}.
\]
The affine reflection at this affine hyperplane is precisely $s_{\alpha,k}$ and the reflections at all such hyperplanes give the action of $W_\ell$. We denote the set of all such affine hyperplanes by $\cH$ and conversely for a hyperplane $H \in \cH$ we denote by $s_H$ the corresponding affine reflection. The statements and definitions in this section can be found in most textbooks on affine reflection groups, e.g. \cite{HumphreysCoxeter}.

\begin{definition}
Consider the complement of all affine hyperplanes in $\cH$
\[
X_\bR^{\mathrm{reg}} = X_\bR \setminus \bigcup_{H \in \cH} H.
\]
A connected component of $X_\bR^{\mathrm{reg}}$ is called an \emph{open alcove}. The closure of a connected component in $X_\bR$ is called a \emph{(closed) alcove}. We denote the set of all (closed) alcoves by $\cA$.
\end{definition}

Note that $W_\ell$ acts simply transitively on the set $\cA$. Points in an open alcove have trivial stabilizer, while points in the boundary of an open alcove have non-trivial stabilizers.

Each affine hyperplane $H\in \cH$ defines two closed halfspaces. For a fixed alcove $A \in \cA$ we denote by $H_A^+$ the half space that contains $A$. Furthermore for two alcoves  $A$ and $A^\prime$ we call a hyperplane $H\in\cH$ \emph{between} $A$ and $A^\prime$ if $H_A^+ \neq H_{A^\prime}^+$. This give rise to the definition of a distance function 
\[
d(A,A^\prime) = \# \{H \in \cH \mid H \text{ between } A \text{ and } A^\prime\},
\]
for $A,A^\prime \in \cA$.

\begin{lemma} \label{lem:distancereduce}
Let $H \in \cH$ be a hyperplane between alcoves $A, A^\prime \in \cA$. Then $d(A,A^\prime) > d(s_HA,A^\prime)$ and $d(A,A^\prime) > d(A,s_HA^\prime)$.
\end{lemma}

For an alcove $A\in\cA$ consider the set of hyperplanes $\cH_A$ that intersect $A$ in maximal dimensions, these are called the walls of $A$. Then $S_A=\{s_H \mid H \in \cH_A\}$ is a generating set of $W_\ell$ as a Coxeter group. With respect to these generators we can see the distance as a choice free substitute for the length function $l$ with respect to $S_A$, namely let $A^\prime$ be another alcove then
\[
d(A,A^\prime)=l(w) \text{ for } wA=A^\prime.
\]

\begin{remark}
With the equality of distance and length, it follows that for a fixed alcove $A$ and $H\in\cH_A$, the hyperplane $H$ is the only hyperplane between $A$ and $s_HA$.
\end{remark}

This leads to the combinatorics of (minimal) galleries.

\begin{definition}
Two alcoves $A,A^\prime\in\cA$ are called adjacent if $A^\prime=s_HA$ for some $H\in\cH_A$. A sequence of alcove $\Gamma=(A_0,A_1,A_2,\ldots,A_r)$ such that $A_i$ is adjacent to $A_{i+1}$ is called an \emph{(alcove) gallery} from $A_0$ to $A_r$. A gallery $\Gamma=(A_0,A_1,A_2,\ldots,A_r)$ such that $r=d(A_0,A_r)$ is called a \emph{minimal gallery}. Note that the set of walls $\{H_i \mid H_i \text{ between } A_{i-1} \text{ and } A_i\}$ is precisely the set of hyperplanes between $A_0$ and $A_r$ in case of a minimal gallery $\Gamma$. We call these hyperplanes the hyperplanes that are crossed by the gallery $\Gamma$.
\end{definition}

Since weights can be contained in a hyperplane, we need a more restrictive notion of hyperplanes between two alcoves.

\begin{definition} \label{def:strictlybetween}
For $\lambda\in X$ denote by $A_\lambda$ an alcove that contains $\lambda$. For $\lambda,\mu \in X$ and choices of alcoves $A_\lambda$ and $A_\mu$, we call a hyperplane $H\in\cH$ \emph{strictly between} $A_\lambda$ and $A_\mu$ if $H$ is between $A_\lambda$ and $A_\mu$ and $H \cap \{\lambda,\mu\} = \varnothing$.
\end{definition}

Note that in Definition \ref{def:strictlybetween}, the choice of an alcove $A_\lambda$ is unique if and only if $\lambda$ is in the interior of an alcove. Otherwise multiple choices are possible. We make frequent use of the following lemma to make a particularly good choice.

\begin{lemma} \label{lem:goodminimalgallery}
Let $\lambda,\mu\in X$ such that $\lambda \in W_\ell \mu$. Then there are choices of alcoves $A_\lambda$ and $A_\mu$ such that a minimal gallery $\Gamma$ from $A_\lambda$ to $A_\mu$ only crosses hyperplanes that are strictly between $A_\lambda$ and $A_\mu$
\end{lemma}
\begin{proof}
Let $A_\lambda$ and $A_\mu$ be any choice for alcoves with $\lambda \in A_\lambda$ and $\mu \in A_\mu$. Let $\Gamma$ be a minimal gallery and assume that there exists a hyperplane $H$ that is between $A_\lambda$ and $A_\mu$ but not strictly between. Without loss of generality assume $\lambda \in H$ (otherwise just rename), hence $\lambda \in s_HA_\lambda$ and so $s_HA_\lambda$ is also a choice for an alcove containing $\lambda$. By Lemma \ref{lem:distancereduce} it holds $d(A_\lambda,A_\mu) > d(s_HA_\lambda,A_\mu)$. Thus we can replace $A_\lambda$ by $s_HA_\lambda$ and the distance between $A_\mu$ and the new $A_\lambda$, i.e. the length of a minimal gallery, strictly decreases. Since the distance is bounded below by $0$ this process has to end after finitely many steps and at that point all hyperplanes crossed by a minimal gallery are strictly between the final choices of $A_\lambda$ and $A_\mu$.  
\end{proof}

\subsection{Affine quantum symmetric pair} \label{sec:quantumsymmpair}

For this section fix $r > 3$. To define the necessary quantum symmetric pairs we consider sets of the form $\bZ/r\bZ$ respectively $\bH/r\bZ$. The restriction of $r>3$ is so to not consider special cases for small $r$. In general one can define similar algebras also for $r=2$ and $r=3$. The relations change slightly in those cases.

We consider the Dynkin diagram of type $\hat{A}_{r-1}$ with indices either labeled by entries in $\bZ/r\bZ$ or $\bH/r\bZ$
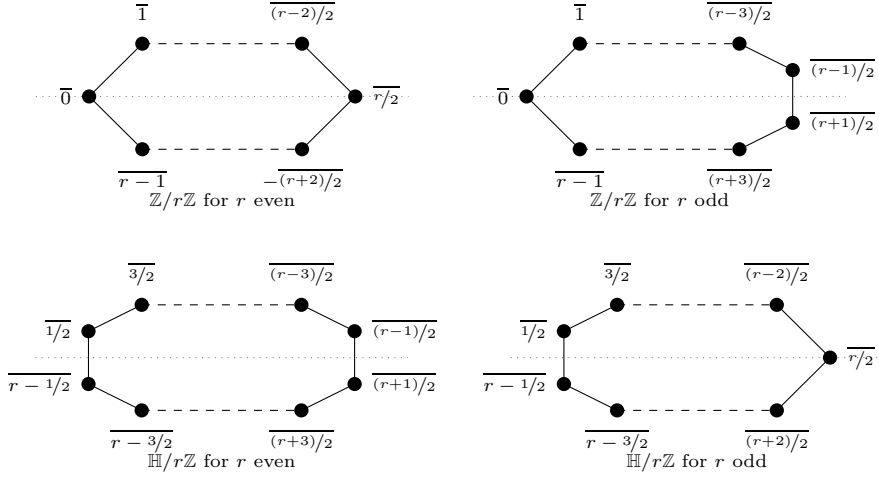
\[
\begin{array}{cc}
\begin{tikzpicture}[anchorbase, scale=.7, tinynodes]
	\draw[dotted, gray] (-1,0) to (6,0);
	\node[circle,fill,inner sep=.5pt,label = {left:$\overline{0}$}] (1) at (0,0){};
	\draw (0,0) to (1,1);
	\draw (0,0) to (1,-1);
	\node[circle,fill,inner sep=.5pt,label = {above:$\overline{1}$}] (2) at (1,1){};
	\node[circle,fill,inner sep=.5pt,label = {below:$\overline{r-1}$}] (3) at (1,-1){};
	\draw[dashed] (1,-1) to (4,-1);
	\draw[dashed] (1,1) to (4,1);
	\node[circle,fill,inner sep=.5pt,label = {above:$\overline{\nicefrac{(r-2)}{2}}$}] (2) at (4,1){};
	\node[circle,fill,inner sep=.5pt,label = {below:$-\overline{\nicefrac{(r+2)}{2}}$}] (3) at (4,-1){};
	\draw (4,1) to (5,0);
	\draw (4,-1) to (5,0);
	\node[circle,fill,inner sep=.5pt,label = {right:$\overline{\nicefrac{r}{2}}$}] (3) at (5,0){};
	\node at (2.5,-2) {$\bZ/r\bZ$ for $r$ even};
\end{tikzpicture}
&
\begin{tikzpicture}[anchorbase, scale=.7, tinynodes]
	\draw[dotted, gray] (-1,0) to (6,0);
	\node[circle,fill,inner sep=.5pt,label = {left:$\overline{0}$}] (1) at (0,0){};
	\draw (0,0) to (1,1);
	\draw (0,0) to (1,-1);
	\node[circle,fill,inner sep=.5pt,label = {above:$\overline{1}$}] (2) at (1,1){};
	\node[circle,fill,inner sep=.5pt,label = {below:$\overline{r-1}$}] (3) at (1,-1){};
	\draw[dashed] (1,-1) to (4,-1);
	\draw[dashed] (1,1) to (4,1);
	\node[circle,fill,inner sep=.5pt,label = {above:$\overline{\nicefrac{(r-3)}{2}}$}] (2) at (4,1){};
	\node[circle,fill,inner sep=.5pt,label = {below:$\overline{\nicefrac{(r+3)}{2}}$}] (3) at (4,-1){};
	\draw (4,1) to (5,.5);
	\draw (4,-1) to (5,-.5);
	\node[circle,fill,inner sep=.5pt,label = {right:$\overline{\nicefrac{(r-1)}{2}}$}] (3) at (5,.5){};
	\node[circle,fill,inner sep=.5pt,label = {right:$\overline{\nicefrac{(r+1)}{2}}$}] (3) at (5,-.5){};
	\draw (5,.5) to (5,-.5);
	\node at (2.5,-2) {$\bZ/r\bZ$ for $r$ odd};
\end{tikzpicture}
\\
&
\\
\begin{tikzpicture}[anchorbase, scale=.7, tinynodes]
	\draw[dotted, gray] (-1,0) to (6,0);
	\node[circle,fill,inner sep=.5pt,label = {left:$\overline{\oneh}$}] (1) at (0,.5){};
	\node[circle,fill,inner sep=.5pt,label = {left:$\overline{r-\oneh}$}] (1) at (0,-.5){};
	\draw (0,.5) to (0,-.5);
	\draw (0,.5) to (1,1);
	\draw (0,-.5) to (1,-1);
	\node[circle,fill,inner sep=.5pt,label = {above:$\overline{\nicefrac{3}{2}}$}] (2) at (1,1){};
	\node[circle,fill,inner sep=.5pt,label = {below:$\overline{r-\nicefrac{3}{2}}$}] (3) at (1,-1){};
	\draw[dashed] (1,-1) to (4,-1);
	\draw[dashed] (1,1) to (4,1);
	\node[circle,fill,inner sep=.5pt,label = {above:$\overline{\nicefrac{(r-3)}{2}}$}] (2) at (4,1){};
	\node[circle,fill,inner sep=.5pt,label = {below:$\overline{\nicefrac{(r+3)}{2}}$}] (3) at (4,-1){};
	\draw (4,1) to (5,.5);
	\draw (4,-1) to (5,-.5);
	\node[circle,fill,inner sep=.5pt,label = {right:$\overline{\nicefrac{(r-1)}{2}}$}] (3) at (5,.5){};
	\node[circle,fill,inner sep=.5pt,label = {right:$\overline{\nicefrac{(r+1)}{2}}$}] (3) at (5,-.5){};
	\draw (5,.5) to (5,-.5);
	\node at (2.5,-2) {$\bH/r\bZ$ for $r$ even};
\end{tikzpicture}
&
\begin{tikzpicture}[anchorbase, scale=.7, tinynodes]
	\draw[dotted, gray] (-1,0) to (6,0);
	\node[circle,fill,inner sep=.5pt,label = {left:$\overline{\oneh}$}] (1) at (0,.5){};
	\node[circle,fill,inner sep=.5pt,label = {left:$\overline{r-\oneh}$}] (1) at (0,-.5){};
	\draw (0,.5) to (0,-.5);
	\draw (0,.5) to (1,1);
	\draw (0,-.5) to (1,-1);
	\node[circle,fill,inner sep=.5pt,label = {above:$\overline{\nicefrac{3}{2}}$}] (2) at (1,1){};
	\node[circle,fill,inner sep=.5pt,label = {below:$\overline{r-\nicefrac{3}{2}}$}] (3) at (1,-1){};
	\draw[dashed] (1,-1) to (4,-1);
	\draw[dashed] (1,1) to (4,1);
	\node[circle,fill,inner sep=.5pt,label = {above:$\overline{\nicefrac{(r-2)}{2}}$}] (2) at (4,1){};
	\node[circle,fill,inner sep=.5pt,label = {below:$\overline{\nicefrac{(r+2)}{2}}$}] (3) at (4,-1){};
	\draw (4,1) to (5,0);
	\draw (4,-1) to (5,0);
	\node[circle,fill,inner sep=.5pt,label = {right:$\overline{\nicefrac{r}{2}}$}] (3) at (5,0){};
	\node at (2.5,-2) {$\bH/r\bZ$ for $r$ odd};
\end{tikzpicture}
\end{array}
\]

For both index sets we consider the automorphism $\Theta$ given by $\Theta(\overline{p}) = -\overline{p}$. Thus in each of the pictured Dynkin diagrams, this is the horizontal reflection along the dotted horizontal line. Depending on whether $r$ is even or odd, and whether considering the index set $\bH/r\bZ$ or $\bZ/r\bZ$, $\Theta$ have between zero and two fixed points on the left or right of the diagram.

In either case of $\bZ/r\bZ$ and $\bH/r\bZ$ we call cosets $\overline{p}$ and $\overline{q}$ linked if $\overline{p} = \overline{q\pm 1}$, i.e. if the corresponding nodes in the affine Dynkin diagram are connected by an edge.

To simplify the notations for the quantum symmetric pair we rewrite one type of quantum Serre relation in form of a non-commutative polynomial
\[
\mathrm{SR}_v(x,y) = x^2y  - [2]_v xyx + yx^2,
\]
for non-commuting variables $x,y$ over $\bQ(v)$. In contrast to the usual quantum group $U_v(\hat{\mathfrak{sl}}_{r})$ the quantum Serre relations for the generators of the quantum symmetric pair depend on the type of the node in the Dynkin diagram of the generator, i.e. they are not invariant under translation of indices.

\begin{definition}
A $\overline{p} \in \bZ/r\bZ$ respectively $\overline{p}\in \bH/r\bZ$ is called 
\begin{enumerate}
\item a fixed index if $\Theta(\overline{p})=\overline{p}$,
\item a $\Theta$-linked index if $\Theta(\overline{p})$ is linked to $\overline{p}$, and
\item a standard index if $\overline{p}$ is neither fixed not $\Theta$-linked.
\end{enumerate}
\end{definition}

The affine quantum symmetric pair is then the following associative algebra.

\begin{definition}{\cite{Kolb}} \label{def:quantumsymmpair}
Let $\rI \in \{\bZ/r\bZ,\bH/r\bZ\}$. The \emph{affine quantum symmetric pair} $B_v(\rI,\Theta)$ is defined to be the associative algebra over $\bQ(v)$ generated by $\{\OPB_{\overline{p}} \mid \overline{p}\in \rI\}$ and $\{\OPL_{\overline{q}}\mid \overline{q}\in \rI$,\, $\overline{q}\neq\Theta(\overline{q})\}$, subject to the following relations: 

For $\overline{p},\overline{q}$ not fixed it holds
\[
\OPL_{\overline{p}} \OPL_{\overline{q}} = \OPL_{\overline{q}} \OPL_{\overline{p}}, \quad \OPL_{\overline{p}}\OPL_{\Theta(\overline{p})} = 1.
\]
For $\overline{p}$ standard, $\overline{r}$ $\Theta$-linked, $\overline{s}$ fixed, and $\overline{q}$ not fixed it holds
\[
\OPL_{\overline{q}}\OPB_{\overline{p}} = \left\lbrace
\begin{array}{rl}
v^2 \OPB_{\overline{p}}\OPL_{\overline{q}} & \text{if } \overline{p}=\overline{q}\\
v^{-2} \OPB_{\overline{p}}\OPL_{\overline{q}} & \text{if } \overline{p}=\Theta(\overline{q}),\\
v^{-1} \OPB_{\overline{p}}\OPL_{\overline{q}} & \text{if } \overline{p},\overline{q} \text{ linked},\\
v \OPB_{\overline{p}}\OPL_{\overline{q}} & \text{if } \overline{p}, \Theta(\overline{q}) \text{ linked},\\
\OPB_{\overline{q}}\OPL_{\overline{q}} & \text{otherwise,} 
\end{array}\right.
\,
\OPL_{\overline{q}}\OPB_{\overline{r}} = \left\lbrace
\begin{array}{rl}
v^3 \OPB_{\overline{r}}\OPL_{\overline{q}} & \text{if } \overline{r}=\overline{q},\\
v^{-3} \OPB_{\overline{r}}\OPL_{\overline{q}} & \text{if } \overline{r}=\Theta(\overline{q}),\\
v^{-1} \OPB_{\overline{r}}\OPL_{\overline{q}} & \text{if } \overline{r}, \overline{q} \text{ linked}, \overline{q} \neq \Theta(\overline{r})\\
v \OPB_{\overline{r}}\OPL_{\overline{q}} & \text{if } \overline{r}, \Theta(\overline{q}) \text{ linked}, \overline{q} \neq \overline{r}\\
\OPB_{\overline{r}}\OPL_{\overline{q}} & \text{otherwise,} 
\end{array}\right.
\]
\[
\OPL_{\overline{q}}\OPB_{\overline{s}} = \OPB_{\overline{s}}\OPL_{\overline{q}}.
\]
The generators $\OPB_{\overline{p}}$ and $\OPB_{\overline{q}}$ commute, unless
\[
\OPB_{\overline{p}}\OPB_{\Theta(\overline{p})}-\OPB_{\Theta(\overline{p})}\OPB_{\overline{p}} = \frac{\OPL_{\overline{p}}-\OPL_{\Theta(\overline{p})}}{v-v^{-1}} \text{ for }  \overline{p} \text{ standard}
\]
or
\[
\mathrm{SR}_v(\OPB_{\overline{p}},\OPB_{\overline{q}}) = 
\left\lbrace
\begin{array}{rl}
0 & \overline{p},\overline{q} \text{ linked, not fixed, and } \Theta(\overline{p})\neq\overline{q},\\
0 & \overline{p},\overline{q} \text{ linked, } \overline{p} \text{ standard, } \overline{q} \text{ fixed},\\
\OPB_{\overline{q}} & \overline{p},\overline{q} \text{ linked, } \overline{p} \text{ fixed, } \overline{q} \text{ standard},\\
-[2]_v \OPB_{\overline{p}} (v\OPL_{\overline{p}} + v^{-2}\OPL_{\Theta(\overline{p})}) & \overline{p},\overline{q} \text{ $\Theta$-linked }, \overline{q}=\Theta(\overline{p}).
\end{array} \right.
\]
\end{definition}

Deriving these definitions from \cite{Kolb} needs some clarifications. The relations between the $\OPB_{\underline{p}}$ are given in \cite[Theorem 7.4]{Kolb}, note that in our language all the $c_i$ that appear in \cite{Kolb} are equal to $1$ and the elements $\mathcal{Z}_j=\OPL_j$ with $j$ an index in the Dynkin diagram. The commutator relations between $\OPB_{\overline{p}}$ and $\OPL_{\overline{q}}$ are found in \cite[(7.7)]{Kolb}.

It is easier to think of $B_v(\rI,\Theta)$ as being an affine analogue of the quantum symmetric pair of type AIII in \cite[Section 7]{Letzter2}. The relations for the quantum symmetric pair of type AIII are nearly the same as for the ordinary quantum group, as in our case. There is either one special generator, corresponding to a fixed index, or two special generators, corresponding to a pair of $\Theta$-fixed indices that do not behave like usual quantum group generators. In contrast to the non-affine case, there are now two such situations. In the Dynkin diagrams above these are the two areas where the dotted line intersects the circle.

\begin{remark} \label{rem:changeingenerators}
Comparing this definition to the ones for non-affine quantum symmetric pairs in \cite[Proposition 7.17 and Proposition 7.18]{ES} (or the idempotent version in \cite{BSWW} in case there are no $\Theta$-linked indices), we see that locally the relations are the same. The pairs $\check{E}_i$ and $\check{F}_i$ in the relations of \cite{ES} are the analogues of $\OPB_{\overline{p}}$ and $\OPB_{\Theta(\overline{p})}$ for $\overline{p}$ standard. Since our generators are ordered in a cyclic way it is not reasonable to choose a ``positive'' and ``negative'' one in such pairs. The generator $\check{B}$ in \cite[Proposition 7.18]{ES} is the analogue of $\OPB_{\overline{p}}$ for a fixed index and the pair of generators $\check{B}_+$ and $\check{B}_-$ in \cite[Proposition 7.17]{ES} are the analogue of $\OPB_{\overline{p}}$ and $\OPB_{\Theta(\overline{p})}$ for $\overline{p}$ being $\Theta$-linked.

In contrast to \cite{ES} we use the analogue of a semi-simple Cartan, while \cite{ES} uses the analogue of a reductive Cartan. This is just for simplicity of notation. One could easily modify the definition and write every generator $\OPB_{\overline{q}}$ as a product of two generators in the vein of the definitions of \cite{ES}.

The only place where the relations differ is the last of the deformed quantum Serre relations in Definition \ref{def:quantumsymmpair}, with $\overline{p}$ being $\Theta$-linked and $\overline{q}=\Theta(\overline{p})$. This is due to the fact that the linear operators $B_{\oneh}$ and $B_{-\oneh}$ in \cite[Definition 7.8]{ES} that give the action of $\check{B}_+$ and $\check{B}_-$ are not symmetric in their definition. This was needed to match the grading of graded category $\cO$ in \cite{ES}, but this is not necessary in our situation.
\end{remark}

\section{Type C} \label{sec:typeC}

For type $C_N$ ($N > 2$) we choose $X=\bigoplus_{i=1}^N \bZ \varepsilon_i$, where the $\varepsilon_i$ are the projection onto the $i$-th diagonal entries for the Cartan subalgebra of diagonal matrices inside $\mathfrak{sp}_{2n}(\bC)$. The $W$-invariant bilinear form $(-,-)$ is given such that the $\varepsilon_i$ form an orthonormal basis. As the positive roots in a root system $\Phi$ in $X$, we choose:
\[
\Phi^+ = \{ \beta_{i,j}^\pm=\varepsilon_i\pm\varepsilon_j \mid 1 \leq i < j \leq N \} \cup \{ \beta_{i}=2\varepsilon_i \mid 1 \leq i \leq N \}.
\]
The simple roots for this choice are $\alpha_i=\beta_{i,i+1}^-$ for $1\leq i < N$ and $\alpha_N=\beta_N$. The corresponding coroots in $X$ are then
\[
(\beta_{i,j}^\pm)^\vee = \beta_{i,j}^\pm \text{ and } \beta_i^\vee = \varepsilon_i.
\]
As defined in Section \ref{sec:quantumgroup} the elements $d_i=1$ for $1 \leq i < N$ and $d_N=2$, since $\alpha_N$ is the only simple long root. This is extended to all positive roots by having $d_\beta=2$ if $\beta=\beta_i$ for some $i$ and $1$ otherwise.

With these choices the fundamental weights are given as $\omega_i = \varepsilon_1 + \ldots + \varepsilon_i$. And thus the dominant integral weights are
\[
X^+ = \{ (\lambda_1,\ldots,\lambda_N) \mid \lambda_i \in \bZ, \lambda_i \geq \lambda_{i+1}, \text{ and } \lambda_N \geq 0\},
\]
where we write row vectors with respect to the basis $\{\varepsilon_i\}_{1\leq i \leq N}$. Then
\[
\rho=\sum_{i=1}^N \omega_i = N\varepsilon_1+(N-1)\varepsilon_2 +\ldots + \varepsilon_N \in X^+
\]
and thus 
\[
X^+_\rho = \{ (\underline{\lambda}_1,\ldots,\underline{\lambda}_N) \mid \underline{\lambda}_i \in \bZ,\, \underline{\lambda}_i > \underline{\lambda}_{i+1}, \text{ and } \underline{\lambda}_N > 0\}.
\]
(Recall the convention that $\underline{\lambda}=\lambda + \rho$ for $\lambda \in X^+$.) 

Since there are roots with $d_\beta \neq 1$ we have to distinguish the case of $\ell$ odd and $\ell$ even as described in Section \ref{sec:quantumgroup}. In case $\ell$ odd, $W_\ell$ is the affine Weyl group of type $C_N$, just scaled by the factor $\ell$. While in case $\ell$ even, $W_\ell$ is the affine Weyl group of type $B_N$ scaled by a factor $\ell$ for the dual root system, i.e. where $\beta_i$ is replaced by $\varepsilon_i$ with coroot $2\varepsilon_i$ instead.

For the action on $[\Uqmod]$ we consider the functor $-\otimes_\bC \Delta_{\mathrm{q}}(\omega_1)$, which is taking the tensor product with the specialization of the quantum analogue of the natural representation.

Since the natural representation is minuscule one obtains the following tensor product decomposition in the generic case.

\begin{proposition}{\cite[Proposition 8.6.3]{HongKang}} \label{prop:tensorproductC}
Let $\lambda \in X^+$. Then in $\Uvmod$ it holds
\[
\Delta_v(\lambda) \otimes \Delta_v(\omega_1) \cong \bigoplus_{i : \lambda + \varepsilon_i \in X^+} \Delta_v(\lambda + \varepsilon_i) \oplus \bigoplus_{i : \lambda - \varepsilon_i \in X^+} \Delta_v(\lambda - \varepsilon_i).
\]
\end{proposition}

In this case the tensor product decomposition is straight forward: If $\lambda + \varepsilon_i$ is dominant the corresponding irreducible module appears and similarly for $\lambda - \varepsilon_i$.

\subsection{Fock space of sequences and operators} \label{sec:fockspaceC}
To embed $\cF(C_N)$ into $\cF_N$, we map $\underline{\lambda}$ to the sequence $\underline{a}_{\underline{\lambda}}$ such that 
\[
\underline{a}_{\underline{\lambda}}(i) = \begin{cases}
1 & \text{if } i \leq 0, \\
1 & \text{if there exists } 1 \leq j \leq N \text{ s.t. } i=\underline{\lambda}_j, \\
0 & \text{otherwise.}
\end{cases}
\]
Note that by construction all non-positive entries in the sequence $\underline{a}$ are $1$ and there are exactly $N$ entries equal to $1$ in the strictly positive part. Hence the sequence $\underline{a}_{\underline{\lambda}}$ has charge $N$. The map is obviously injective. By abuse of notation we simply write $\underline{\lambda}$ for the sequence as well and identify $\cF(C_N)$ with its image.

Looking at the linear operators from Definition \ref{def:movingoperators}, we immediately see the following.

\begin{lemma} \label{lem:nearlystable}
Let $\underline{\lambda} \in \cF(C_N)$ and $r\in\bH$. Then $\re_r \underline{\lambda} \in \cF(C_N)$ and furthermore $\re_r \underline{\lambda}$ is non-zero if and only if $\lambda - \varepsilon_i$ is dominant for $r=\underline{\lambda}_i-\oneh$. 

Similarly, $\rf_r \underline{\lambda} \in \cF(C_N)$, for $r \neq \oneh$, and $\lambda + \varepsilon_i$ is dominant if and only if $\rf_r \underline{\lambda}$ is non-zero for $r=\underline{\lambda}_i+\oneh$.
\end{lemma}

Note that all operators $\re_r$ preserve the subspace $\cF(C_N)$ with $\re_r$ acting as zero for all $r \leq \oneh$. On the other hand the operators $\rf_r$ keep the subspace invariant, except for $\rf_{\oneh}$. The obvious reason being that $\rf_{\oneh}$ can move a $1$ from position $0$ to position $1$, leaving the subspace $\cF(C_N)$.

\subsection{Linkage and operators} \label{sec:linkagetypeC}

We are now considering when two weights obtained from $\underline{\lambda}$ by applying two of the operators from above are linked under the assumption that the sequences stay inside $\cF(C_N)$.

\begin{remark}
In the following we often just say: assume $\underline{\mu}=\re_r\underline{\lambda}$ is defined. By this we mean that $\re_r\underline{\lambda}$ is again a sequence coming from the embedding of $\cF(C_N)$.
\end{remark}

Even though the roots of type $C$ are different from type $A$, the following Lemma has pretty much the same proof as in type $A$, as the roots of the form $\beta_{ij}^-$ are the only ones playing a role.

\begin{lemma} \label{lem:linkedEorFs}
Let $\underline{\lambda} \in X^+_\rho$. Assume $\underline{\mu}=\re_r\underline{\lambda}$ and $\underline{\nu}=\re_s\underline{\lambda}$, for $r \neq s$ are defined. Then $\mu$ and $\nu$ are linked if and only if $r \in s + \ell \bZ$.

Assume $\underline{\mu}=\rf_r\underline{\lambda}$ and $\underline{\nu}=\rf_s\underline{\lambda}$, for $r \neq s$, are defined. Then $\mu$ and $\nu$ are linked if and only if $r \in s + \ell \bZ$.
\end{lemma}
\begin{proof}
It holds $\underline{\mu}=\underline{\lambda}-\varepsilon_i$ where $i$ is determined by $\underline{\lambda}_i=r+\oneh$. Similarly $\re_s$ defines $j$ such that $\underline{\lambda}_j=s+\oneh$. Since $s \neq r$ we can assume that $i < j$.

Assume now that the two weights are linked. Using Lemma \ref{lem:goodminimalgallery}, fix alcoves $A_{\underline{\mu}}$ and $A_{\underline{\nu}}$ such that a minimal gallery $A_{\underline{\mu}}=A_0,\ldots,A_t=A_{\underline{\nu}}$ only crosses walls that are strictly between $A_{\underline{\mu}}$ and $A_{\underline{\nu}}$.

It holds
\[
( \underline{\nu}-\underline{\mu},\beta^\vee) = \begin{cases}
2 & \text{if } \beta=\beta_{i,j}^-, \\
1 & \text{if } \beta \in \{\beta_{i,k}^-, \beta_{i,k}^+,\beta_{k,i}^+ \mid k \neq j\} \cup \{\beta_i\},\\
-1 & \text{if } \beta \in \{\beta_{k,i}^- \mid k < i\},\\
0 & \text{otherwise.}
\end{cases}
\]
Since we assume that $\underline{\mu} \neq \underline{\nu}$, but $\mu$ and $\nu$ are linked, the minimal gallery needs to have at least length $1$. Thus there exists a hyperplane strictly between $A_{\underline{\lambda}}$ and $A_{\underline{\mu}}$, otherwise the two alcoves agree which is a contradiction to the simply transitive action of $W_\ell$ on the set of alcoves.

Looking at the values $(\underline{\nu}-\underline{\mu},\beta^\vee)$ above, the only root that can have an affine hyperplane strictly between $A_{\underline{\mu}}$ and $A_{\underline{\nu}}$ is $\beta=\beta^-_{ij}$. Hence the length of the minimal gallery is exactly $1$ and it holds $\underline{\nu} = s_{\beta_{i,j}^-,m}\underline{\mu}$ for some $m\in\bZ$. Let $H=H_{\beta^-_{ij},m}$ be this hyperplane. Hence
\[
1 + m \ell=( \underline{\mu},(\beta_{i,j}^-)^\vee) = \underline{\lambda}_i - \underline{\lambda}_j + 1.
\]
So we obtain that $\underline{\lambda}_i = \underline{\lambda}_j + m\ell$ and so $r \in s +\ell \bZ$.

Now assume that $r \in s +\ell \bZ$. Then $\underline{\lambda}_i-\underline{\lambda}_j = m\ell$ for some $m\in\bZ$. Hence $(\underline{\nu},(\beta_{ij}^-)^\vee) = 1 + m\ell$ and thus $s_{\beta_{ij}^-,m}(\underline{\nu})=\underline{\mu}$ and so the weights $\mu$ and $\nu$ are linked.

The statement for the $\rf$-operators is done completely analogous.
\end{proof}

Lemma \ref{lem:linkedEorFs} can be used word for word to see that in type $A$ linked weights are obtained by operators with indices congruent modulo $\ell$.

\begin{lemma}\label{lem:linkedEFdifferentstart}
Let $\underline{\lambda} \in X^+_\rho$. Assume $\underline{\mu}=\re_r\underline{\lambda}$ and $\underline{\nu}=\rf_s\underline{\lambda}$, for $r+1 \neq s$, are defined. Then $\mu$ and $\nu$ are linked if and only if $r \in -s + \ell \bZ$.
\end{lemma}
\begin{proof}
It holds $\underline{\mu}=\underline{\lambda}-\varepsilon_i$ where $i$ is determined by $\underline{\lambda}_i=r+\oneh$. Similarly $\rf_s$ defines $j$ such that $\underline{\lambda}_j=s-\oneh$. Since $s \neq r+1$ it holds that $i \neq j$.

Assume that the weights are linked. Fix $A_{\underline{\mu}}$ and $A_{\underline{\nu}}$ via Lemma \ref{lem:goodminimalgallery} such that the minimal gallery only crosses hyperplanes that are strictly between the two alcoves.

Then we check that the only possibility for $|(\underline{\mu}-\underline{\nu},\beta^\vee)| > 1$ is $(\underline{\mu}-\underline{\nu},\beta^\vee)=2$ for the choice $\beta = \beta_{i,j}^+$ (Technically $\beta_{i,j}^+$ is only defined for $i<j$, so for $i>j$ we just set $\beta_{i,j}^+=\beta_{j,i}^+$). Thus the unique affine hyperplane strictly between $A_{\underline{\mu}}$ and $A_{\underline{\nu}}$ is of the form $H_{\beta_{i,j}^+,m}$ and so it holds $\underline{\nu} = s_{\beta_{i,j}^+,m}\underline{\mu}$. In formulas for the weights we thus get
\[
1 + m \ell=(\underline{\mu},(\beta_{i,j}^+)^\vee) = \underline{\lambda}_i + \underline{\lambda}_j + 1.
\]
So we obtain $\underline{\lambda}_i = -\underline{\lambda}_j + m\ell$ and thus so $r \in -s + \ell \bZ$.

Assume that $r \in -s + \ell \bZ$, then $\underline{\lambda}_i+\underline{\lambda}_j = m\ell$ for some $m\in\bZ$. Hence $(\underline{\nu},(\beta_{ij}^+)^\vee) = 1 + m\ell$ and thus $s_{\beta_{ij}^+,m}(\underline{\nu})=\underline{\mu}$ and so the weights $\mu$ and $\nu$ are linked.
\end{proof}

This is the main difference to type $A$. The relationship of the form $r \in -s + \ell \bZ$ forces us to fix a type of origin. Hence the embedding of $\cF(C_N)$ depends on $N$ in this case.

\begin{lemma}\label{lem:linkedEFsamestartodd}
Let $\ell$ be odd and $\underline{\lambda} \in X^+_\rho$. Assume $\underline{\mu}=\re_{r-1}\underline{\lambda}$ and $\underline{\nu}=\rf_{r}\underline{\lambda}$ are defined. Then $\mu$ and $\nu$ are linked if and only if $r \in \oneh + \ell \bZ$.
\end{lemma}
\begin{proof}
In this case there exists $i$ such that $\underline{\mu}=\underline{\lambda}-\varepsilon_i$ and $\underline{\nu}=\underline{\lambda}+\varepsilon_i$ for $i$ determined by $\underline{\lambda}_i=r-\oneh$.

Assume that the weights are linked and fix alcoves $A_{\underline{\mu}}$ and $A_{\underline{\nu}}$ such that a minimal gallery only crosses hyperplanes strictly between the alcoves, using Lemma \ref{lem:goodminimalgallery}.

Checking the values of $(\underline{\nu}-\underline{\mu},\beta^\vee)$, there are multiple choices for hyperplanes that can be strictly between the chosen alcoves. Namely we have
\[
(\underline{\nu}-\underline{\mu},\beta^\vee) = \begin{cases}
2 & \text{if } \beta=\beta_{i}, \\
2 & \text{if } \beta\in\{\beta_{i,j}^-,\beta_{i,j}^+,\beta_{j,i}^+ \mid j \neq i\}, \\
-2 & \text{if } \beta\in\{\beta_{j,i}^- \mid j \neq i\},\\
0 & \text{otherwise.}
\end{cases}
\]
Denote by $H_1=H_{\gamma_1,m_1},\ldots,H_t=H_{\gamma_t,m_t}$ be the hyperplanes crossed by the minimal gallery. By construction each $H_i$ is strictly between $A_{\underline{\mu}}$ and $A_{\underline{\nu}}$, hence for each $\gamma_i$ it holds $(\underline{\mu}-\underline{\nu},\gamma_i^\vee)=\pm 2$ and especially $( \underline{\lambda},\gamma_i^\vee) = m_i\ell$, i.e. $\underline{\lambda}\in H_i$.

We now go through the different cases for $\gamma_1$.

\textbf{Case: $\gamma_1=\beta_i$:} In this case it holds that
\[
1+m_1\ell=(\underline{\nu},\beta_i^\vee) = \underline{\lambda}_i+1.
\]
Thus $\underline{\lambda}_i \in \ell\bZ$ and so $r\in \oneh +\ell\bZ$. Since in this case it immediately holds $s_{\beta_i,m_1}\underline{\nu}=\underline{\mu}$ it follows that $s_{\beta_i,m_1}A_{\underline{\nu}} = A_{\underline{\mu}}$ and so the minimal gallery had length 1.

\textbf{Case: $\gamma_1=\beta_{i,j}^-$ for some $j > i$:} In this case the equation is
\[
1+m_1\ell=(\underline{\nu},(\beta_{i,j}^-)^\vee)= \underline{\lambda}_i-\underline{\lambda}_j+1.
\]
Thus $\underline{\lambda}_i=\underline{\lambda}_j+m_1\ell$ and it holds $s_{\beta_{i,j}^-,m_1}\underline{\nu}=\underline{\lambda}+\varepsilon_j$. For every hyperplane $H_q$ for $q>1$ and $\gamma_q \neq \beta_{i,j}^+$ it holds $m_q\ell=(\underline{\lambda},\gamma_p^\vee) =(\underline{\lambda}+\varepsilon_j,\gamma_p^\vee)$, hence $\underline{\lambda}+\varepsilon_j \in H_q$. Thus applying $s_{\gamma_q,m_q}$ leaves $\underline{\lambda}+\varepsilon_j$ invariant. Since $\underline{\lambda}+\varepsilon_j \neq \mu$, there must exists $p>1$ such that $\gamma_p=\beta_{i,j}^+$. In which case
\[
1+m_p\ell=(\underline{\lambda}+\varepsilon_i,(\beta_{i,j}^+)^\vee)=(\underline{\lambda}+\varepsilon_j,(\beta_{i,j}^+)^\vee) = \underline{\lambda}_i+\underline{\lambda}_j+1.
\]
Hence $\underline{\lambda}_i=-\underline{\lambda}_j+m_p\ell$. Thus $2\underline{\lambda}_i = (m_p+m_1)\ell$. Since $2\underline{\lambda}_i$ is an even integer and $\ell$ is odd, $(m_p+m_1)$ is even and so $\underline{\lambda}_i \in \ell\bZ$ and equivalently $r \in \oneh+\ell\bZ$.

\textbf{Case: $\gamma_1=\beta_{i,j}^+$ for some $j > i$:} This is nearly identical to the previous case. The only difference is that one first obtains $\underline{\lambda}_i\in -\underline{\lambda}_j+\ell\bZ$, applying the first reflection gives $\underline{\lambda}-\varepsilon_j$, and the second used hyperplane is for $\beta_{i,j}^-$ and one obtains $\underline{\lambda}_i\in \underline{\lambda}_j+\ell\bZ$. The rest of the argument is then the same.

\textbf{Case: $\gamma_1=\beta_{j,i}^-$ for some $j < i$:} In this case we start with 
\[
-1+m_1\ell=( \underline{\nu},(\beta_{j,i}^-)^\vee) = \underline{\lambda}_j-\underline{\lambda}_i-1.
\]
Hence $\underline{\lambda}_i\in \underline{\lambda}_j + \ell\bZ$ and $s_{\gamma_1,m_1}(\underline{\nu})=\underline{\lambda}+\varepsilon_j$ and the rest is as in the cases before.

\textbf{Case: $\gamma_1=\beta_{j,i}^+$ for some $j < i$:} This similarly follows as the previous cases.

For the converse, assume $r \in \oneh + \ell \bZ$. Then $\underline{\lambda}_i = m \ell$ for some $m\in\bZ$. Hence $(\underline{\nu},(\beta_{i})^\vee) = 1 + m\ell$ and thus $s_{\beta_{i},m}(\underline{\nu})=\underline{\mu}$ and so the weights $\mu$ and $\nu$ are linked.
\end{proof}

For the case of $\ell$ odd Lemma \ref{lem:linkedEFsamestartodd}, gives a special case of Lemma \ref{lem:linkedEFdifferentstart} with a slightly more delicate proof. This special case will lead to generators for the quantum symmetric pair of non-standard indices. The special case for $\ell$ even is handled in the next statement.

\begin{lemma}\label{lem:linkedEFsamestarteven}
Let $\ell$ be even and $\underline{\lambda} \in X^+_\rho$. Assume $\underline{\mu}=\re_{r-1}\underline{\lambda}$ and $\underline{\nu}=\rf_{r}\underline{\lambda}$ are defined. Then $\mu$ and $\nu$ are linked if and only if $r \in \oneh + (\ellh)\bZ$.
\end{lemma}
\begin{proof}
The arguments in the proof follow the ones for Lemma \ref{lem:linkedEFsamestartodd} so we only point out the differences.

\textbf{Case: $\gamma_1 = \beta_i$:} In contrast to Lemma \ref{lem:linkedEFsamestartodd} we now obtain that $\underline{\lambda}_i \in (\ellh)\bZ$ and so $r\in \oneh +(\ellh)\bZ$. This is of course due to $\ell_{\beta_i}=\ellh$.

In all other cases we obtain $2\underline{\lambda}_i \in \ell\bZ$ as before, but now we can simply divide $\ell$ by $2$ and obtain $\underline{\lambda}_i \in (\ellh)\bZ$, which in turn implies $r \in \oneh + (\ellh)\bZ$.

For the converse, assume $r \in \oneh + (\ellh)\bZ$. Then $\underline{\lambda}_i = m \ellh$ for some $m\in\bZ$. Hence $(\underline{\nu},(\beta_{i})^\vee) = 1 + m\ellh$ and thus $s_{\beta_{i},m}(\underline{\nu})=\underline{\mu}$ since $\ell_{\beta_i}=\ellh$ and so the weights $\mu$ and $\nu$ are linked.
\end{proof}

In contrast to Lemma \ref{lem:linkedEFsamestartodd}, in Lemma \ref{lem:linkedEFsamestarteven} we see that for $\ell$ even the situation that $\re_{r-1}$ and $\rf_{r}$ produce linked weights happens for two types of positions, for a $1$ at a position in $\ell\bZ$ or at a position in $\ellh+\ell\bZ$. This is reflected in the existence of two $\Theta$-linked pairs of indices in the Dynkin diagram for $\bH/\ell\bH$ and $\ell$ even in Section \ref{sec:quantumsymmpair}.

\subsection{Quantum symmetric pair action} \label{sec:quantumsympairactionC}

We now define operators on $\cF_N$ as follows. Recall the automorphism $\Theta:\bH/\ell\bZ \rightarrow \bH/\ell\bZ$ from Section \ref{sec:quantumsymmpair} that changes the sign.

\begin{definition} \label{def:quantumsymoperatorsC}
Let $\underline{a} \in \mathcal{S}_{\bZ}$ and $\overline{p} \in \bH/\ell\bZ$. For $\Theta(\overline{p})\neq \overline{p}$ we define
\begin{align*}
\OPB_{\overline{p}}\, \underline{a} &= v^{T_\ell^{e-f}(\Theta(\overline{p}),\underline{a})}\sum_{j \in \overline{p}} v^{R_\ell^{e-f}(j,\underline{a})} \re_j\,\underline{a} + \sum_{j \in -\overline{p}} v^{L_\ell^{f-e}(j,\underline{a})} \rf_j\,\underline{a} \text{ and }\\
\OPL_{\overline{p}} \underline{a} &= v^{T_\ell^{f-e}(\overline{p},\underline{a})}v^{T_\ell^{e-f}(\Theta(\overline{p}),\underline{a})} \underline{a}.
\end{align*}
For $\Theta(\overline{p})=\overline{p}$, i.e. $\overline{p}=\overline{\ellh}$, we define
\[
\OPB_{\overline{\ellh}} \underline{a} = v^{-1} v^{T_\ell^{e-f}(\overline{\ellh},\underline{a})}\sum_{j \in \overline{\ellh}} v^{R_\ell^{e-f}(j,\underline{a})} \re_j\underline{a} + \sum_{j \in \overline{\ellh}} v^{L_\ell^{f-e}(j,\underline{a})} \rf_j\underline{a}.
\]
This defines linear operators on $\cF$ that restrict to $\cF_N$ for any $N$.
\end{definition}

As mentioned at the very end of Section \ref{sec:typeA}, the action of $U_v(\hat{\mathfrak{sl}}_\ell)^\prime$ is defined on any $\cF_N$ with the same definitions as for $\cF_0$. Thus by comparing we get the following.

\begin{lemma} \label{lem:comparetypeA}
Let $\underline{a} \in \mathcal{S}_{\bZ}$ and $\overline{p} \in \bH/\ell\bZ$. For $\Theta(\overline{p})\neq \overline{p}$
\[
\OPB_{\overline{p}}\, \underline{a} = \OPE_{\overline{p}}\OPK_{-\overline{p}}^{-1}\,\underline{a} + \OPF_{-\overline{p}}\,\underline{a} \quad \text{ and } \quad \OPL_{\overline{p}} \underline{a} = \OPK_{\overline{p}}\OPK_{-\overline{p}}^{-1} \underline{a}
\]
and
\[
\OPB_{\overline{\ellh}}\, \underline{a} = v^{-1}\OPE_{\overline{\ellh}}\OPK_{\overline{\ellh}}^{-1}\,\underline{a} + \OPF_{\overline{\ellh}}\,\underline{a}
\]
\end{lemma}

The definition of the linear operators $\OPB_{\overline{p}}$ is the reason why we do not embed into the Fock space of charge $0$ (as one does in type $A$). The linear operator $\OPB_{\overline{p}}$ involves the $\re_i$'s for $i \in \overline{p}$, but also $\rf_j$'s for $-j\in\overline{p}$. Thus the definition is not invariant under arbitrary translation as the linear operators for the type $A$ action, only under translation by multiples of $\ell$.

\begin{remark}
Note that for $\underline{\lambda} \in \cF(C_N)$ it holds $\OPB_{\overline{p}} \underline{\lambda} \in \cF(C_N)$ if $-\oneh \notin \overline{p}$. For $\OPB_{\overline{\ellh}}$ one just needs to consider the summand using $\rf_{\oneh}$ that can produce a sequence that is not contained in $\cF(C_N)$. This is exactly parallel to the situation in type $A_{N-1}$ where one can apply an operator $\OPF_{\overline{p}}$, that in the language of Young diagrams, creates a box in row $N+1$, thus leaving the span of polynomial weights for $\mathfrak{gl}_N$.
\end{remark}

For the relations of the operators we then get the following.

\begin{proposition} \label{prop:operator_relations_C}
The linear operators $\{\OPB_{\overline{p}} \mid \overline{p} \in \bH/\ell\bZ\}$ and $\{\OPL_{\overline{p}} \mid \overline{p} \in \bH/\ell\bZ, \Theta(\overline{p})\neq\overline{p}\}$ satisfy the relations given in Definition \ref{def:quantumsymmpair} by substituting the operators for the generators with the same name.
\end{proposition}
\begin{proof}
Since we did not specify how to make all the choices in \cite{Kolb}, we give a short sketch of how to quickly check that the relations are indeed satisfied.

The relations between the $\OPL_{\overline{p}}$ are immediate by definition. They all multiply a basis element with a fixed scalar and the scalars for $\OPL_{\overline{p}}$ and $\OPL_{-\overline{p}}$ are inverse to each other.

The commuting relations between $\OPL_{\overline{q}}$ and $\OPB_{\overline{p}}$ are a straight-forward calculation using Lemma \ref{lem:comparetypeA} above.

That $\OPB_{\overline{p}}$ and $\OPB_{\overline{q}}$ commute unless for the specified index choices is also immediate from the definition. Either the summands of $\OPB_{\overline{p}}$ and $\OPB_{\overline{q}}$ modify a basis vector at positions that are not neighboured, i.e. the cosets are not linked and so the summands commute. Or, in case that $\Theta(\overline{p})$ and $\overline{q}$ are linked, their summands can modify the same position by moving a $1$ into different directions or moving two $1$'s into the same position. In this case the product of the summands is just always zero, hence they commute.

The commutator relation between $\OPB_{\overline{p}}$ and $\OPB_{\Theta(\overline{p})}$ is a direct and simple computation using Lemma \ref{lem:comparetypeA}.

Finally the deformed and non-deformed Serre relations follow from \cite{ES} (or precisely their use of \cite{Letzter2}), since for a linked pair of indices ${\overline{p}}$ and ${\overline{q}}$ the Serre relation are independent of the rest of the Dynkin diagram, the calculation only involves the relation between the standard generators of the quantized enveloping algebra for the indices ${\overline{p}}$, ${\overline{q}}$, ${\Theta(\overline{p})}$, and ${\Theta(\overline{q})}$, which are local and hence the same in the affine case and in \cite{ES}, since $\ell > 3$.
\end{proof}

Thus putting everything together we obtain the statement.

\begin{theorem} \label{thm:actionanddecompC}
There exists an action of $B_v(\bH/\ell\bZ,\Theta)$ on $\cF_N$ such that for $\lambda \in X^+$ the decomposition of $[\Delta_{\mathrm{q}}(\lambda) \otimes \Delta_{\mathrm{q}}(\omega_1)]$ in $[\Uqmod]$ with respect to the classes of Weyl modules is obtained from
\[
\sum_{\overline{p}\in \bH/\ell\bZ} \OPB_{\overline{p}}\underline{\lambda}
\]
by projecting onto the subspace $\cF(C_N)$ and evaluating $v=1$. Furthermore if $[\Delta_{\mathrm{q}}(\mu)]$ and $[\Delta_{\mathrm{q}}(\nu)]$ appear in the decomposition with $\mu$ and $\nu$ linked, then there exists a unique $\overline{p}\in \bH/\ell\bZ$ such that $\underline{\mu}$ and $\underline{\nu}$ appear in $\OPB_{\overline{p}}\underline{\lambda}$.
\end{theorem}
\begin{proof}
The action is the one from Proposition \ref{prop:operator_relations_C}. That the decomposition is given by the sum of all operators is Proposition \ref{prop:tensorproductC} together with Proposition \ref{prop:tensorproductdecomp} to obtain the translation of the decomposition into weight combinatorics and then Lemma \ref{lem:nearlystable} together with the definition of the operators themselves.

That classes of two Weyl modules with linked weights are obtained from a unique operator $\OPB_{\overline{p}}$ is then Lemmas \ref{lem:linkedEorFs}, \ref{lem:linkedEFdifferentstart}, \ref{lem:linkedEFsamestarteven}, and \ref{lem:linkedEFsamestartodd}, .
\end{proof}

We want to address now the question of embedding into a single Fock space for different ranks.

In type $A$, every $\cF(A_N)$ can be embedded into $\cF_0$. This was natural on a combinatorial level due to the choice of $\mathfrak{gl}$ instead of $\mathfrak{sl}$ and hence the ability to choose the element $\rho$. The choice was made such that the weight $0$ always gets mapped to the sequence with all $1$'s in the non-positive half and $0$'s in the positive half. From the point of view of the affine operators and the action, any shift of the origin, i.e. an embedding into a different $\cF_k$ makes no difference, since the operators only involves entries that are congruent mod $\ell$. Hence an embedding into a different $\cF_k$ is just related to a shift in the used operators. On the Lie theory side this would just correspond to a different choice of $\rho$.

We can make a similar construction in type $C$ if we restrict to certain $N$. Although this is more artificial since on the Lie theory side, the element $\rho$ is fixed. 

\begin{definition} \label{def:shiftsequences}
Let $N = m \ell + k$ for $m \geq 0$ and $0 \leq k < \ell$. Then we define a map from $\cF_N$ to $\cF_k$ via $\underline{a}^{(m)}(i) = \underline{a}(i+m\ell)$ for $\underline{a} \in \mathcal{S}_{\bZ,N}$. We call $\underline{a}^{(m)}$ the sequence shifted by $m$ $\ell$-steps.

In this way we can view $\cF(C_{N})$ as a subspace of $\cF_k$, which we call the shifted embedding of $\cF(C_{N})$.
\end{definition}

That the map in Definition \ref{def:shiftsequences} is well-defined follows immediately with a simple calculation for the charge. Since it intertwines the action of $\re_i$ and $\re_{i-m\ell}$ (and similar of $\rf_i$ and $\rf_{i-m\ell}$), the affine operators $\OPB_{\overline{p}}$ and $\OPL_{\overline{p}}$ defined on $\cF_N$ and $\cF_k$ as restrictions from $\cF$ commute with the map of shifting a sequence by $m$ steps. In terms of weights this is the same as looking at the weight $\lambda + \rho - m\ell(1,\ldots,1)$ which for most $m$ is not dominant.

Thus using this embedding for a fixed $0 \leq k < \ell$ we can formulate this as follows.

\begin{proposition} \label{prop:operatorsasdecomp}
Let $\underline{a} \in \mathcal{S}_{\bZ,k}$ and $r \geq 0$. Then for $m \gg 0$, $\underline{a}=\underline{\lambda}^{(m)}$ for $\lambda$ a dominant weight for $U_{\mathrm{q}}$ of type $C_{m\ell+k}$. Furthermore there are Laurent polynomials $d_{\lambda,\mu}(v)$ with non-negative integer coefficients for $\mu$ dominant for $U_q$ such that 
\[
\sum_{(\bH/\ell\bZ)^r} \OPB_{\overline{p}_1} \cdots \OPB_{\overline{p}_r}\underline{\lambda}^{(m)} = \sum_ {\mu \text{ dominant for }U_q} d_{\lambda,\mu}(v)\underline{\mu}^{(m)},
\]
and
\[
[\Delta_{\mathrm{q}}(\lambda) \otimes \Delta_{\mathrm{q}}(\omega_1)^{\otimes r}] = \sum_ {\mu \text{ dominant for }U_q} d_{\lambda,\mu}(1)[\Delta_{\mathrm{q}}(\mu)].
\]
\end{proposition}
\begin{proof}
For $\underline{a}$ we consider the left most $0$ in the non-positive part of the sequence. If this zero is at position $-n^\prime$ then using operators of the form $\re_i$ this can be reduced to a sequence $\underline{b}$ that has no $0$'s in the non-positive part and exactly $k$ $1$'s at the positions $1,\ldots,k$. Then $\underline{b}$ is the image of $\underline{0}$ under the shifted embedding of $\cF(C_{m^\prime\ell+k})$ for $m^\prime \ell > n^\prime$. Then by applying operators $\rf_i$ in the reverse order to the sequence of operators $\re_i$ before one obtains a dominant weight $\lambda$ with $\underline{a}=\underline{\lambda}^{(m^\prime)}$. If there was no $0$ in the non-positive part then one can just use $m^\prime=0$.

A dominant weight $\lambda$ for $U_{\mathrm{q}}$ of rank $m^\prime\ell+k$ can be viewed as a dominant weight for $U_{\mathrm{q}}$ of rank $m\ell+k$ for $m>m^\prime$, by just filling it up with $0$ at the last $(m-m^\prime)\ell$ entries. Choose $m>m^\prime$ such that $m\ell > n^\prime + r$. By our discussion about which $\rf_i$ can leave the embedded subspace $\cF_N$, after Lemma \ref{lem:nearlystable}, we see that no product $\OPB_{\overline{p}_1} \cdots \OPB_{\overline{p}_r}$ applied to $\underline{a}$ can create a weight that is not in the shifted embedding of $\cF(C_{m\ell + k})$, since at most the $1$'s at the position $-n^\prime-1,\ldots, -n^\prime-r$ can be moved to the right. Hence the statement then follows from Theorem \ref{thm:actionanddecompC}.
\end{proof}

Note that the proof of Proposition \ref{prop:operator_relations_B} gives a clear bound for what $m$ needs to be, it is just necessary that $m\ell > n^\prime + r$ for $n^\prime$ the position of the left most entry $0$.

\begin{remark}
Note that the definition of the operators $\OPB_{\overline{p}}$ is not uniquely determined. As already mentioned in Remark \ref{rem:changeingenerators}, one can change some coefficients in the definition of the generators and obtain a slightly modified algebra. 

Since the definition in type $A$ is not unique as well, compare for example the definition in \cite[Theorem 3.1]{RT} and in \cite[Section 10.1]{Ariki}, one can make similar modifications in type $C$.
\end{remark}

\section{Type B} \label{sec:typeB}

In type $B_N$ ($N>1$), we choose $\varepsilon_i \in \fh^\ast$ the projection onto the $i$-th diagonal entries for the Cartan subalgebra of diagonal matrices inside $\mathfrak{so}_{2n+1}(\bC)$. The $W$-invariant bilinear form $(-,-)$ is given such that the $\varepsilon_i$ form an orthonormal basis. For the positive roots we choose 
\[
\Phi^+ = \{ \beta_{i,j}^\pm=\varepsilon_i\pm\varepsilon_j \mid 1 \leq i < j \leq N \} \cup \{ \beta_{i}=\varepsilon_i \mid 1 \leq i \leq N \}
\]
and similar to type $C$, the simple roots are $\alpha_i=\beta_{i,i+1}^-$ for $1\leq i < N$, but here $\alpha_N=\beta_N$. The corresponding coroots are
\[
(\beta_{i,j}^\pm)^\vee = \beta_{i,j}^\pm \text{ and } \beta_i^\vee = 2\varepsilon_i.
\]
Following Section \ref{sec:quantumgroup} we have $d_i=2$ for $1 \leq i < N$ and $d_N=1$, since $\alpha_N$ is the only simple short root in this case. This gets extended to all positive roots via $d_\beta=2$ for $\beta=\beta_i$ for some $i$ and $d_\beta=1$ otherwise.

The noticeable change is in the integral weight lattice. The fundamental weights are $\omega_i = \varepsilon_1 + \ldots + \varepsilon_i$ for $i < N$ and $\omega_N=\oneh(\varepsilon_1+\ldots+\varepsilon_N)$. Hence for the integral weights $X=\bigoplus_{i+1}^N \bZ \omega_i$, the dominant weights $X^+$ that can be naturally divided into two subsets
\begin{align*}
X^{\oneh,+} &= \{ (\lambda_1,\ldots,\lambda_N) \mid \lambda_i \in \bZ, \lambda_i \geq \lambda_{i+1}, \text{ and } \lambda_N \geq 0\} \text{ and}\\
X^{1,+} &= \{ (\lambda_1,\ldots,\lambda_N) \mid \lambda_i \in \bH, \lambda_i \geq \lambda_{i+1}, \text{ and } \lambda_N \geq 0\},
\end{align*}
written as row vectors with respect to $\{\varepsilon_i\}_{1\leq i\leq N}$ inside $\fh^\ast$. We call $X^{\oneh,+}$ the integer weights, not to be confused with the integral weights, and $X^{1,+}$ the half-integer weights.

Summing up the fundamental weights we get 
\[
\rho=\sum_{i=1}^N \omega_i = (N-\oneh)\varepsilon_1+(N-\threeh)\varepsilon_2 +\ldots + \oneh\varepsilon_N \in X^{\oneh,+}.
\]
Thus adding $\rho$ we obtain the sets to define sequences, namely
\begin{align*}
X^{\oneh,+}_\rho &= \{ (\underline{\lambda}_1,\ldots,\underline{\lambda}_N) \mid \underline{\lambda}_i \in \bH,\, \underline{\lambda}_i > \underline{\lambda}_{i+1}, \text{ and } \underline{\lambda}_N > 0\},\\
X^{1,+}_\rho &= \{ (\underline{\lambda}_1,\ldots,\underline{\lambda}_N) \mid \underline{\lambda}_i \in \bZ,\, \underline{\lambda}_i > \underline{\lambda}_{i+1}, \text{ and } \underline{\lambda}_N > 0\},
\end{align*}
again recalling that $\underline{\lambda}=\lambda + \rho$ for $\lambda \in X^+$.

\begin{remark}
This makes clear our convention of naming the set of integer weights $X^{\oneh,+}$ and the set of half-integer weights $X^{1,+}$ (and not the other way around). After adding $\rho$, elements in $X^{\oneh,+}_\rho$ have entries in $\bH$ and elements of $X^{1,+}_\rho$ have entries in $\bZ$. Thus the naming convention makes it easy to recognise what the domains for the sequences in the different cases are.
\end{remark}

Again the parity of $\ell$ is important. For $\ell$ odd, $W_\ell$ is the affine Weyl group of type $B_N$, scaled by the factor $\ell$. In case $\ell$ even, $W_\ell$ is the affine Weyl group of type $C_N$ scaled by a factor $\ell$ for the dual root system, i.e. where one would replace $\beta_{ij}^\pm$ in the root system by $\oneh\beta_{ij}^\pm$.

For the action on $\Uqmod$ we use the functor $-\otimes_\bC \Delta_{\mathrm{q}}(\omega_1)$, again the tensor product with the the specialization of the quantum analogue of the natural representation.

\begin{remark}
Note that in contrast to the case of $U_v(\fg)$ of type $A$ or $C$, $\Delta_v(\omega_1)$ is not a tensor generator of $\Uvmod$, hence there are other possible finite dimensional representations, like the specialization of the spin representation $\Delta_{\mathrm{q}}(\omega_N)$, that can give interesting actions on $\Uqmod$.
\end{remark}

In type $B$ the natural representation is not minuscule. Hence the tensor product decomposition has a slight complication.

\begin{proposition}{\cite[Proposition 8.6.3]{HongKang}} \label{prop:tensorproductB}
For $\lambda \in X^{+}$ it holds in $\Uvmod$
\[
\Delta_v(\lambda) \otimes \Delta_v(\omega_1) \cong \bigoplus_{i : \lambda + \varepsilon_i \in X^+} \Delta_v(\lambda + \varepsilon_i) \oplus \bigoplus_{i : \lambda - \varepsilon_i \in X^+} \Delta_v(\lambda - \varepsilon_i) \oplus \Delta_v(\lambda)^{\oplus \delta_{\mathrm{pos}}},
\]
where $\delta_{\mathrm{pos}}=0$ if $\lambda_N=0$ and $\delta_{\mathrm{pos}}=1$ if $\lambda_N > 0$.
\end{proposition}

Thus in type $B$ the tensor product decomposition for weights in $X^{1,+}$ is straight-forward with the only difference to type $C$ being the appearance of a summand $\Delta_v(\lambda)$ itself. Even the sequences are looking very similar to the type $C$ case.

For weights in $X^{\oneh,+}$ on the other hand, the tensor product decomposition rule depends on the explicit weight of the Weyl module, namely the summand $\Delta_v(\lambda)$ only appears if $\lambda_N>0$. In addition we are dealing with sequences on half-integers in this case.

\subsection{Fock space of sequences and operators} \label{sec:fockspaceB}
We decompose the Fock space as
\[
\cF(B_N) = \cF^1(B_N) \oplus \cF^{\oneh}(B_N),
\]
with $\cF^1(B_N)$ having basis $\underline{\lambda}\in X^{1,+}_\rho$ and $\cF^{\oneh}(B_N)$ having basis $\underline{\lambda}\in X^{\oneh,+}_\rho$. For $\cF^1(B_N)$ we can use the same definition of sequences and Fock spaces as in type $C$, and embed $\cF^1(B_N)$ into $\cF_N$ as in Section \ref{sec:fockspaceC}. 

For the case of $\cF^{\oneh}(B_N)$ we use the sequences $\cS_{\bH}$ and corresponding Fock space $\cF^{\oneh}$ from Section \ref{sec:fockspace}. Hence we map $\underline{\lambda}\in X^{\oneh,+}_\rho$ to the sequence $\underline{a}_{\underline{\lambda}}$ with
\[
\underline{a}_{\underline{\lambda}}(i) = \begin{cases}
1 & \text{if } i < 0, \\
1 & \text{if there exists } 1 \leq j \leq N \text{ s.t. } i=\underline{\lambda}_j, \\
0 & \text{otherwise}
\end{cases}
\]
to embed $\cF^{\oneh}(B_N)$ into $\cF^{\oneh}$. As before it is clear that the image is contained in the subspace $\cF^{\oneh}_N$ of sequences of charge $N$. We continue to write $\underline{\lambda}$ for the sequence as well.

The analogue of Lemma \ref{lem:nearlystable} holds in type $B$ as well, i.e. $\re_i$ preserves the subpsaces $\cF^1(B_N)$ and $\cF^{\oneh}(B_N)$, $\rf_i$ does so for $i \notin \{0,\oneh\}$, and $\rf_i$ does not preserve the subpsaces for $i \in \{0,\oneh\}$.

In contrast to type $C$, the cases of $\ell$ odd and even have to be treated separately.

\subsection{Linkage and operators for odd $\ell$}

Throughout this section we assume that 
\[
\ell > 3 \text{ is odd.}
\]
The proofs in this section have a very similar flavour to the ones in Section \ref{sec:linkagetypeC}. Unfortunately there are some subtle differences. We start with the statements that work for both the integer and the half-integer weight cases.

\begin{lemma} \label{lem:linkedEorFtypeBodd}
Let $\underline{\lambda} \in X^{+}_\rho$ and $\ell$ odd. Assume that $\underline{\mu}=\re_r\underline{\lambda}$ and $\underline{\nu}=\re_s\underline{\lambda}$ are defined for $r \neq s$. Then $\mu$ and $\nu$ are linked if and only if $r \in s + \ell \bZ$.

Assume that $\underline{\mu^\prime}=\rf_r\underline{\lambda}$ and $\underline{\nu^\prime}=\rf_s\underline{\lambda}$ are defined for $r \neq s$. Then $\mu^\prime$ and $\nu^\prime$ are linked if and only if $r \in s + \ell \bZ$.
\end{lemma}
\begin{proof}
As in Lemma \ref{lem:linkedEorFs}, $\underline{\mu}=\underline{\lambda}-\varepsilon_i$ with $\underline{\lambda}_i=r+\oneh$ and $\underline{\nu}=\underline{\lambda}-\varepsilon_j$ with $\underline{\lambda}_j=s+\oneh$ and we assume that $i<j$.

Assume now that $\underline{\mu}$ and $\underline{\nu}$ are linked. Use Lemma \ref{lem:goodminimalgallery} to obtain alcoves $A_{\underline{\mu}}$ and $A_{\underline{\nu}}$ with a minimal gallery only crossing walls strictly between the alcoves. The absolute value of $(\underline{\nu}-\underline{\mu},\beta^\vee)$ can be $2$ in case of $\beta \in \{\beta_{ij}^-,\beta_i,\beta_j\}$. In case that a hyperplane of the form $H_{\beta_{ij}^-,m}$ is strictly between the alcoves, it follows as in Lemma \ref{lem:linkedEorFs} that $r \in s+\ell\bZ$.

Hence assume that there is no hyperplane of the form $H_{\beta_{ij}^-,m}$ strictly between the alcoves. Since $\underline{\nu}-\underline{\mu}=\varepsilon_i-\varepsilon_j$ there has to be both a hyperplane of the form $H_{\beta_{i},m^\prime}$ and of the form $H_{\beta_{j},m^{\prime\prime}}$ strictly between the alcoves since we assume that the weights are linked. This implies $2\underline{\lambda}_i = m^\prime \ell$ and $2\underline{\lambda}_j = m^{\prime\prime}\ell$. Hence $2(\underline{\lambda}_i-\underline{\lambda}_j)=(m^\prime-m^{\prime\prime}) \ell$. Since $\ell$ is odd $(m^\prime-m^{\prime\prime})$ has to be even and so $H_{\beta_{ij}^-,(m^\prime-m^{\prime\prime})/2}$ is strictly between the alcoves, a contradiction.

Conversely assume that $r \in s + \ell \bZ$. Then $\underline{\lambda}_i-\underline{\lambda}_j$ is a multiple of $\ell$. Hence $(\underline{\lambda},(\beta_{ij}^-)^\vee)=\underline{\lambda}_i-\underline{\lambda}_j$ is a multiple of $\ell$ and so $(\underline{\nu},(\beta_{ij}^-)^\vee) = 1 + m\ell$ for some $m\in\bZ$. Thus $s_{\beta_{ij}^-,m}(\underline{\nu})=\underline{\mu}$ and so the weights are linked.

The analogous statement for $\rf_i$ and $\rf_j$ holds with the analogous proof.
\end{proof}

For the case of mixed operators we get accordingly.

\begin{lemma} \label{lem:linkedEFtypedifferentBodd}
Let $\underline{\lambda} \in X^{+}_\rho$ and $\ell$ odd. Assume that $\underline{\mu}=\re_r\underline{\lambda}$ and $\underline{\nu}=\rf_s\underline{\lambda}$ are defined for $r+1 \neq s$. Then $\mu$ and $\nu$ are linked if and only if $r \in -s + \ell \bZ$.
\end{lemma}
\begin{proof}
The proof is nearly identical to Lemma \ref{lem:linkedEorFtypeBodd}. Fix $i \neq j$ via $\underline{\mu}=\underline{\lambda}-\varepsilon_i$ and $\underline{\nu}=\underline{\lambda}+\varepsilon_j$. Then the roots that give absolute values strictly bigger than $1$ for $(\underline{\nu}-\underline{\mu},\beta^\vee)=(\varepsilon_i+\varepsilon_j,\beta^\vee)$ are $\beta \in \{\beta_{ij}^+,\beta_i,\beta_j\}$. 

Assume that the weights are linked. Fixing the alcoves as usual and the corresponding minimal gallery, we see that if $H_{\beta_{ij}^+,m}$ is strictly between the alcoves for some $m$ then the statement follows immediately. If one assumes that such a hyperplane does not exist, the contradiction is obtained as in the proof of Lemma \ref{lem:linkedEorFtypeBodd}.

If it holds $r \in -s + \ell \bZ$, then $(\underline{\lambda},(\beta_{ij}^+)^\vee)=\underline{\lambda}_i+\underline{\lambda}_j$ is a multiple of $\ell$ and so $(\underline{\nu},(\beta_{ij}^-)^\vee) = 1 + m\ell$ for some $m\in\bZ$. Thus $s_{\beta_{ij}^+,m}(\underline{\nu})=\underline{\mu}$ and so the weights are linked.
\end{proof}

For the analogue of Lemmas \ref{lem:linkedEFsamestartodd} and \ref{lem:linkedEFsamestarteven} we get the following.

\begin{lemma} \label{lem:linkedEFtypesameBodd}
Let $\ell$ odd. Assume $\underline{\lambda} \in X^{1,+}_\rho$ and that $\underline{\mu}=\re_{r-1}\underline{\lambda}$ and $\underline{\nu}=\rf_{r}\underline{\lambda}$ are defined. Then $\mu$ and $\nu$ are linked if and only if $r \in \oneh + \ell \bZ$.

Assume $\underline{\lambda} \in X^{\oneh,+}_\rho$ and that $\underline{\mu}=\re_{r-1}\underline{\lambda}$ and $\underline{\nu}=\rf_{r}\underline{\lambda}$ are defined. Then $\mu$ and $\nu$ are linked if and only if $r \in \ellph + \ell \bZ$.
\end{lemma}
\begin{proof}
Fix $i$ via $\underline{\mu}=\underline{\lambda}-\varepsilon_i$ respectively $\underline{\nu}=\underline{\lambda}+\varepsilon_i$, i.e. with $\underline{\lambda}_i=r-\oneh$. 

Assume that $\mu$ and $\nu$ are linked. Again use Lemma \ref{lem:goodminimalgallery} to obtain alcoves $A_{\underline{\mu}}$ and $A_{\underline{\nu}}$ with a minimal gallery only crossing walls strictly between the alcoves. The possible positive roots such that the absolute value $|(\underline{\nu}-\underline{\mu},\beta^\vee)|>1$ are $\beta=\beta_i$ for the value $4$ and for the value $2$ it is $\beta =\beta_{ij}^\pm$ for $i<j$ or $\beta =\beta_{ji}^\pm$ for $j<i$ .

Let $A_0,\ldots,A_t$ be the alcoves in the minimal gallery connecting $A_{\underline{\mu}}$ and $A_{\underline{\nu}}$ and $H_1,\ldots,H_t$ the crossed hyperplanes in order, all strictly between the alcoves by Lemma \ref{lem:goodminimalgallery}. Fix positive roots and integers such that $H_p=H_{\gamma_p,m_p}$. Recall that for all these it holds $|(\underline{\nu}-\underline{\mu},\gamma_p^\vee)|>1$.

\textbf{Case: $\gamma_1=\beta_i$ and $r=1$:} Then $(\underline{\nu},\beta_i^\vee) = 2 +m_1\ell$ since we need to have $s_{\beta_i,m_1}(\underline{\nu})=\underline{\mu}$. Thus $2\underline{\lambda}_i+2=2+m_1\ell$. In case that $\underline{\lambda} \in X^{1,+}_\rho$, it holds that $m_1$ must be even (since $\ell$ is odd) and so $\underline{\lambda}_i \in \ell\bZ$ or equivalently $r \in \oneh + \ell \bZ$. If on the other hand $\underline{\lambda} \in X^{\oneh,+}_\rho$, then $m_1$ must be odd and we obtain $\underline{\lambda}_i \in \ellh +\ell\bZ$ or equivalently $r \in \ellph + \ell \bZ$.

\textbf{Case: $\gamma_1=\beta_{ij}^-$ for some $j$:} Then it holds that $(\underline{\nu},(\beta_{ij}^-)^\vee) = 1 +m_1\ell$, $s_{\beta_{ij}^-,m_1}(\underline{\nu})=\underline{\lambda} + \varepsilon_j \neq \underline{\mu}$, and $\underline{\lambda}_i - \underline{\lambda}_j = m_1\ell$. Now for $q>1$ it holds $(\underline{\lambda} + \varepsilon_j,\gamma_q^\vee) \in \ell\bZ$ and thus $s_{\gamma_q,m_q}(\underline{\lambda} + \varepsilon_j)=\underline{\lambda} + \varepsilon_j$ unless $\gamma_q \in \{ \beta_{i},\beta_{ij}^+\}$. Hence either $\beta_{i}$ or $\beta_{ij}^+$ has to appear as a $\gamma_q$ for $q>1$. Let $p>1$ be the minimal such that $\gamma_p \in \{ \beta_{i},\beta_{ij}^+\}$.

If $\gamma_p=\beta_{ij}^+$ then
\[
(\underline{\lambda}+\varepsilon_j,(\beta_{ij}^+)^\vee)=(\underline{\nu},(\beta_{ij}^+)^\vee) = 1 + m_p \ell.
\]
Hence $\underline{\lambda}_i + \underline{\lambda}_j = m_p \ell$ and so $2\underline{\lambda}_i = (m_1+m_p)\ell$. As above, for $\underline{\lambda} \in X^{1,+}_\rho$, it follows that $r \in \oneh + \ell \bZ$ and for $\underline{\lambda} \in X^{\oneh,+}_\rho$ it follows that $r \in \ellph + \ell \bZ$. Furthermore, we see that $\gamma_q\neq \beta_i$ for all $q$ since we already have $s_{\beta_{ij}^+,m^\prime}(\underline{\lambda} + \varepsilon_j) = \underline{\mu}$ in this case.

If $\gamma_p=\beta_{i}$ then
\[ 
(\underline{\lambda}+\varepsilon_j,(\beta_{i})^\vee)= (\underline{\nu},(\beta_{i})^\vee)-2.
\]
But $H_{\beta_{i},m_p}$ is assumed to be strictly between the alcoves, thus $(\underline{\nu},(\beta_{i})^\vee)=k + m_p\ell$ for $k \in \{1,2,3\}$ ($k=4$ is not possible since $\underline{\mu}$ would lie on $H_{\beta_{i},m_p}$ in that case). Then $(\underline{\lambda}+\varepsilon_j,(\beta_{i})^\vee) = k-2 + m_p\ell$. 
\begin{enumerate}
\item If $k=1$, then $H_{\beta_{i},m_p}$ was already crossed, which cannot be the case. 
\item If $k=2$ then $\underline{\lambda}+\varepsilon_j$ is on $H_{\beta_{i},m_p}$ and invariant under $s_{\beta_{i},m_p}$, hence $\beta_{ij}^+$ has to appear in the sequence of $\gamma$'s that follow and one argues as above. 
\item If $k=3$, it holds $s_{\beta_i,m_p}(\underline{\lambda}+\varepsilon_j) = \underline{\lambda}+\varepsilon_j-\varepsilon_i$. But $(\underline{\lambda}+\varepsilon_j-\varepsilon_i,\gamma_q^\vee)=(\underline{\mu},\gamma_q^\vee)$ for $q>p$ except for $\gamma_q=\beta_{ij}^+$ and for $\beta_{ij}^+$ it holds $(\underline{\lambda}+\varepsilon_j-\varepsilon_i,(\beta_{ij}^+)^\vee)=(\underline{\lambda},(\beta_{ij}^+)^\vee)$. Hence $\underline{\lambda}+\varepsilon_j-\varepsilon_i$ is already in the same half-space as $\underline{\mu}$ for all $H_q$ with $q>p$ and $\gamma_q \neq \beta_{ij}^+$ and it is on the hyperplane $H_q$ for $q>p$ and $\gamma_q = \beta_{ij}^+$. Thus the only hyperplane left to cross is $H_{p+1}$ and it must hold that $\gamma_{p+1} = \beta_{ij}^+$, but $s_{\beta_{ij}^+,m_{p+1}}(\underline{\lambda}+\varepsilon_j-\varepsilon_i)=\underline{\lambda}+\varepsilon_j-\varepsilon_i \neq \underline{\mu}$, which is a contradiction to the construction of the gallery.
\end{enumerate}

\textbf{Remaining cases:} The remaining cases for $\gamma_1$ are all done in the same way as the previous case with slight modifications in the signs that appear, but the same general arguments.

Conversely for $\underline{\lambda} \in X^{1,+}_\rho$ and $r \in \oneh +\ell\bZ$, $\underline{\lambda}_i = m\ell$ for some $m$. Hence $(\underline{\lambda}+\varepsilon_i,\beta_i^\vee) = 2\underline{\lambda}_i+2=2+2m\ell$. Thus $s_{\beta_i,2m}(\underline{\nu})=\underline{\mu}$ and so $\nu$ and $\mu$ are linked.

For $\underline{\lambda} \in X^{\oneh,+}_\rho$ and $r \in \ellph +\ell\bZ$, $\underline{\lambda}_i = \ellh + m\ell$. Hence $(\underline{\lambda}+\varepsilon_i,\beta_i^\vee) = 2\underline{\lambda}_i+2=2+(2m+1)\ell$. Thus $s_{\beta_i,2m+1}(\underline{\nu})=\underline{\mu}$ and so $\nu$ and $\mu$ are linked.
\end{proof}

In contrast to type $C$, $\lambda$ can be linked to a weight $\mu$ obtained by applying an operator.

\begin{lemma} \label{lem:linkedEIdtypeBodd1}
Let $\underline{\lambda} \in X^{1,+}_\rho$ and $\ell$ is odd. Assume that $\underline{\mu}=\re_r\underline{\lambda}$ is defined. Then $\mu$ and $\lambda$ are linked if and only if $r \in \ellh +\ell\bZ$.

Assume that $\underline{\nu}=\rf_r\underline{\lambda}$ is defined. Then $\nu$ and $\lambda$ are linked if and only if $r \in \ellh +\ell\bZ$.
\end{lemma}
\begin{proof}
Fix $i$ via $\underline{\mu}=\underline{\lambda}-\varepsilon_i$ with $\underline{\lambda}_i=r+\oneh$. 

Assume now that $\mu$ and $\lambda$ are linked. The absolute value of $(\underline{\lambda}-\underline{\mu},\beta^\vee)$ can only be $2$ for $\beta =\beta_i$. Hence there exists a hyperplane $H_{\beta_i,m}$ strictly between and so $(\underline{\lambda},\beta_i^\vee)=2\underline{\lambda}_i=1 + m\ell$. Since $2\underline{\lambda}_i$ is an even integer and $\ell$ is odd, $m$ is odd as well. So we get $r = \ellh + \frac{m-1}{2}\ell \in \ellh +\ell\bZ$.

If $r \in \ellh +\ell\bZ$ then $\underline{\lambda}_i= \ellph + m\ell$ for some $m\in\bZ$. Hence $(\underline{\lambda},\beta_i^\vee) = 1 + (2m+1)\ell$ and so $\lambda$ and $\mu$ are linked.

The case of $\underline{\nu}$ follows, considering that in this case $\re_r\underline{\nu}=\underline{\lambda}$.
\end{proof}

The proof for the next Lemma is completely analogous to Lemma \ref{lem:linkedEIdtypeBodd1}.

\begin{lemma} \label{lem:linkedEIdtypeBodd2}
Let $\underline{\lambda} \in X^{\oneh,+}_\rho$ and $\ell$ is odd. Assume that $\underline{\mu}=\re_r\underline{\lambda}$ is defined. Then $\mu$ and $\lambda$ are linked if and only if $r \in \ell\bZ$.

Assume that $\underline{\nu}=\rf_r\underline{\lambda}$ is defined. Then $\nu$ and $\lambda$ are linked if and only if $r \in \ell\bZ$.
\end{lemma}
%
%
%

\subsection{Quantum symmetric pair action for odd $\ell$} \label{sec:quantumsympairactionBodd}
As in the previous section we are assuming throughout the section that
\[
\ell > 3 \text{ is odd.}
\]

Recall the various operators and counting statistics from Section \ref{sec:operatorsandstats}. In contrast to type $C$, we need to use operators for both types of index sets.

The definition of the linear operators is very close to the ones of type $C$ in Section \ref{sec:quantumsympairactionC}. with one difference for the elements of $\bH/\ell\bZ$ respectively $\bZ/\ell\bZ$ that are fixed by $\Theta$ as defined in Section \ref{sec:quantumsymmpair}. Note that for $\ell$ odd, $\Theta(\overline{i})=\overline{i}$ implies that $\overline{i}=\overline{\ellh}$ for $\overline{i} \in \bH/\ell\bZ$ and $\overline{i}=\overline{0}$ for $\overline{i} \in \bZ/\ell\bZ$. This corresponds to the two fixed nodes in the odd Dynkin diagrams in Section \ref{sec:quantumsymmpair}.

\begin{definition} \label{def:quantumsymoperatorsB}
Assume $\ell$ odd. For $\underline{a} \in \mathcal{S}_{\bZ}$ let $\overline{p} \in \bH/\ell\bZ$ and for $\underline{a} \in \mathcal{S}_{\bH}$ let $\overline{p} \in \bZ/\ell\bZ$. If $\Theta(\overline{p})\neq\overline{p}$ define
\begin{align*}
\OPB_{\overline{p}}\, \underline{a} &= v^{T_\ell^{e-f}(-\overline{p},\underline{a})}\sum_{j \in \overline{p}} v^{R_\ell^{e-f}(j,\underline{a})} \re_j\,\underline{a} + \sum_{j \in -\overline{p}} v^{L_\ell^{f-e}(j,\underline{a})} \rf_j\,\underline{a}.\\
\OPL_{\overline{p}} \underline{a} &= v^{T_\ell^{f-e}(\overline{p},\underline{a})}v^{T_\ell^{e-f}(-\overline{p},\underline{a})} \underline{a}.
\end{align*}
If $\Theta(\overline{p})=\overline{p}$ define
\[
\OPB_{\overline{p}} \underline{a} = v^{-1} v^{T_\ell^{e-f}(\overline{p},\underline{a})}\sum_{j \in \overline{p}} v^{R_\ell^{e-f}(j,\underline{a})} \re_j\underline{a} + \sum_{j \in \overline{p}} v^{L_\ell^{f-e}(j,\underline{a})} \rf_j\underline{a} + v^{T_\ell^{e-f}(\overline{p},\underline{a})}\underline{a}.
\]
Finally, for $\underline{a} \in \mathcal{S}_{\bH}$ and $z \in \oneh + \ell\bZ$ define
\[
\OPB_{\overline{0}}^{[z]} \underline{a} = v^{-1} v^{T_\ell^{e-f}(\overline{0},\underline{a})}\sum_{j \in \overline{0}} v^{R_\ell^{e-f}(j,\underline{a})} \re_j\underline{a} + \sum_{j \in \overline{0}} v^{L_\ell^{f-e}(j,\underline{a})} \rf_j\underline{a} + \delta_{\underline{a}(z),0} v^{T_\ell^{e-f}(\overline{0},\underline{a})}\underline{a},
\]
where $\delta_{\underline{a}(z),0}=1-\underline{a}(z)$. This defines linear operators on $\cF$ and $\cF^{\oneh}$ that restrict to $\cF_N$ and $\cF_N^{\oneh}$ for any $N$.
\end{definition}

The linear operator $\OPB_{\overline{0}}^{[z]}$ is not contained in the image of the affine quantum symmetric pair. It is needed to describe the relationship with the tensor product decomposition. Comparing this with the statements in Lemma \ref{lem:comparetypeA}, the same identifications hold except for the case $\Theta(\underline{p})=\underline{p}$.

\begin{lemma} \label{lem:comparetypeAB}
Let $\underline{a} \in \mathcal{S}_{\bZ}$ (respectively $\underline{a} \in \mathcal{S}_{\bH}$) and $\overline{p} \in \bH/\ell\bZ$ (respectively $\overline{p} \in \bZ/\ell\bZ$) with $\Theta(\underline{i})=\underline{i}$ then
\[
\OPB_{\overline{p}}\, \underline{a} = v^{-1}\OPE_{\overline{p}}\OPK_{\overline{p}}^{-1}\,\underline{a} + \OPF_{\overline{p}}\,\underline{a} + \OPK_{\overline{p}}^{-1}\underline{a}
\]
and for $\underline{a} \in \mathcal{S}_{\bH}$ and $z\in\oneh+\ell\bZ$
\[
\OPB_{\overline{p}}^{[z]} \, \underline{a} = v^{-1}\OPE_{\overline{p}}\OPK_{\overline{p}}^{-1}\,\underline{a} + \OPF_{\overline{p}}\,\underline{a} + \delta_{\underline{a}(z),0}\OPK_{\overline{p}}^{-1}\underline{a}
\]
\end{lemma}

\begin{remark}
In \cite{Kolb} a quantum symmetric pair using a generator of the form $v^{-1}\OPE_{\overline{p}}\OPK_{\overline{p}}^{-1} + \OPF_{\overline{p}} + \OPK_{\overline{p}}^{-1}$ is called a non-standard quantum symmetric pair. While the one that uses only generators of the form that appeared in type $C$, i.e. without the extra summand $\OPK_{\overline{p}}^{-1}$ are called standard quantum symmetric pairs.
\end{remark}

Acting with the generators on $\cF^{1}(B_N)$ respectively $\cF^{1}(B_N)$ has the same possibility of having a single summand that is not in the subspace.

\begin{remark}
In case of $\cF_N$ the operators $\OPB_{\overline{i}}$ leave the subspace $\cF^1(B_N)$ invariant, except for $-\oneh \in \overline{i}$. For the operator $\OPB_{\overline{\oneh}}$ the reason is again that it contains a summand $\rf_{\oneh}$ that does not leave $\cF^1(B_N)$ invariant. While for $\cF_N^{\oneh}$ the operators $\OPB_{\overline{0}}$ and $\OPB_{\overline{0}}^{[\oneh]}$ do not leave the subspace invariant because of the summand $\rf_{0}$.
\end{remark}

For the relations of the operators we then get the following.

\begin{proposition} \label{prop:operator_relations_B}
Assume $\ell$ odd and $\rI \in \{\bZ/\ell\bZ,\bH/\ell\bZ\}$. The linear operators $\{\OPB_{\overline{p}} \mid \overline{p} \in \rI\}$ and $\{\OPL_{\overline{p}} \mid \overline{p} \in \rI, \Theta(\overline{p})\neq\overline{p}\}$ satisfy the relations of $B_v(\rI,\Theta)$ given in Definition \ref{def:quantumsymmpair} by substituting the operators for the generators with the same name.
\end{proposition}
\begin{proof}
All relations in Definition \ref{def:quantumsymmpair} not involving the operator $\OPB_{\overline{p}}$ for $\overline{p}$ fixed are precisely the same as in type $C$ and thus hold by Proposition \ref{prop:operator_relations_C}. Note that for $\overline{p}$ fixed, the new operator $\OPB_{\overline{p}}$ only differs from the description in Lemma \ref{lem:comparetypeA} by adding a term of the form $\OPK_{\overline{p}}^{-1}$. The commutator relation with all $\OPL_{\overline{q}}$ is then obvious. The commutator relation with ``distant'' $\OPB_{\overline{q}}$ is also obvious. Hence the only one left are the non-trivial quantum Serre relations. There are only two that involve a fixed index, which are both a simple few line calculation to verify.
\end{proof}

As in type $C$ this all combines then to the following. For the case of weights in $X^{1,+}$ the statement is the complete analogue to Theorem \ref{thm:actionanddecompC} with precisely the same proof.

\begin{theorem} \label{thm:actionanddecompB1}
Assume $\ell$ odd. There exists an action of $B_v(\bH/\ell\bZ,\Theta)$ on $\cF_N$ such that for $\lambda \in X^{1,+}$ the decomposition of $[\Delta_{\mathrm{q}}(\lambda) \otimes \Delta_{\mathrm{q}}(\omega_1)]$ in $[\Uqmod]$ with respect to the classes of Weyl modules is obtained from
\[
\sum_{\overline{p}\in \bH/\ell\bZ} \OPB_{\overline{p}}\underline{\lambda}
\]
by projecting onto the subspace $\cF(C_N)$ and evaluating $v=1$. 

Furthermore if $[\Delta_{\mathrm{q}}(\mu)]$ and $[\Delta_{\mathrm{q}}(\nu)]$ appear in the decomposition for $\mu$ and $\nu$ linked, then there exists a unique $\overline{p}\in \bH/\ell\bZ$ such that $\underline{\mu}$ and $\underline{\nu}$ appear in $\OPB_{\overline{p}}\underline{\lambda}$.
\end{theorem}

In case of $\lambda \in X^{\oneh,+}$, the dependence of the tensor product decomposition rule on $\lambda_N$ makes this less clean.

\begin{theorem} \label{thm:actionanddecompB2}
Assume $\ell$ odd. There exists an action of $B_v(\bZ/\ell\bZ,\Theta)$ on $\cF_N^{\oneh}$ such that for $\lambda \in X^{\oneh,+}$ the decomposition of $[\Delta_{\mathrm{q}}(\lambda) \otimes \Delta_{\mathrm{q}}(\omega_1)]$ in $[\Uqmod]$ with respect to the classes of Weyl modules is obtained from
\[
\sum_{\overline{p}\in \bZ/\ell\bZ} \OPB_{\overline{p}}\underline{\lambda}
\]
if $\lambda_N>0$ and from
\[
\sum_{\overline{p}\in \bZ/\ell\bZ,\overline{p}\neq\overline{0}} \OPB_{\overline{p}}\underline{\lambda} + \OPB_{\overline{0}}^{[\oneh]}\underline{\lambda}
\]
if $\lambda_N=0$, by projecting onto the subspace $\cF(C_N)$ and evaluating $v=1$. 

If $[\Delta_{\mathrm{q}}(\mu)]$ and $[\Delta_{\mathrm{q}}(\nu)]$ appear in the decomposition for $\mu$ and $\nu$ linked, then there exists a unique $\overline{p}\in \bZ/\ell\bZ$ such that $\underline{\mu}$ and $\underline{\nu}$ appear in $\OPB_{\overline{p}}\underline{\lambda}$.
\end{theorem}
\begin{proof}
The only difference to Theorem \ref{thm:actionanddecompB1} is that in case of $\lambda \in X^{\oneh,+}$ the class of the tensor product $[\Delta_{\mathrm{q}}(\lambda) \otimes \Delta_{\mathrm{q}}(\omega_1)]$ contains a summand $[\Delta_{\mathrm{q}}(\lambda)]$ if and only if $\lambda_N>0$, But $\OPB_{\overline{0}}\underline{\lambda}$ always contain a non-zero multiple of $\underline{\lambda}$ as a summand, hence one has to apply $\OPB_{\overline{0}}^{[\oneh]}$ instead in case $\lambda_N=0$.
\end{proof}

Let $0 \leq k < \ell$ and $N=m\ell + k$ for $m \geq 0$. One can use the map from Definition \ref{def:shiftsequences} to embed $\cF^{1}(B_{m\ell+k})$ into $\cF^{1}_k$. In this case the analogue of Proposition \ref{prop:operatorsasdecomp} follows immediately.

In case of $\underline{a}\in\cS_{\bH,N}$ the map from Definition \ref{def:shiftsequences} embeds $\cF^{\oneh}(B_{m\ell+k})$ into $\cF^{\oneh}_k$. To formulate the corresponding statement, define $\OPB_{\overline{p}}^{[z]}=\OPB_{\overline{p}}$ if $\overline{p}\neq\overline{0}$ for $z\in\oneh+\ell\bZ$.

\begin{proposition} \label{prop:operatorsasdecompBhint}
Let $\underline{a} \in \mathcal{S}_{\bH,k}$ and $r \geq 0$. Then for $m \gg 0$, $\underline{a}=\underline{\lambda}^{(m)}$ for $\lambda \in X^{\oneh,+}$ for $U_{\mathrm{q}}$ of type $B_{m\ell+k}$. Furthermore there are Laurent polynomials $d_{\lambda,\mu}(v)$ with non-negative integer coefficients for $\mu \in X^{\oneh,+}$ such that 
\[
\sum_{(\bZ/\ell\bZ)^r} \OPB_{\overline{p}_1}^{[z]} \cdots \OPB_{\overline{p}_r}^{[z]}\underline{\lambda}^{(m)} = \sum_ {\mu \text{ dominant for }U_{\mathrm{q}}} d_{\lambda,\mu}(v)\underline{\mu}^{(m)},
\]
with $z=\oneh-m\ell$ and
\[
[\Delta_{\mathrm{q}}(\lambda) \otimes \Delta_{\mathrm{q}}(\omega_1)^{\otimes r}] = \sum_ {\mu \text{ dominant for }U_{\mathrm{q}}} d_{\lambda,\mu}(1)[\Delta_{\mathrm{q}}(\mu)].
\]
\end{proposition}
\begin{proof}
Note that the only difference to Proposition \ref{prop:operatorsasdecomp} is that instead of $\OPB_{\overline{0}}$ one has to use $\OPB_{\overline{0}}^{[z]}$. We already know that $\sum_{(\bZ/\ell\bZ)^r} \OPB_{\overline{p}_1} \cdots \OPB_{\overline{p}_r}\underline{\lambda}^{(m)}$ is contained in the shifted embedding of $\cF^{\oneh}(B_{m\ell+k})$ for $m \gg 0$, hence also the summand $\sum_{(\bZ/\ell\bZ)^r} \OPB_{\overline{p}_1}^{[z]} \cdots \OPB_{\overline{p}_r}^{[z]}\underline{\lambda}^{(m)}$.

In case $m \gg 0$ all weights $\nu$ such that $[\Delta_{\mathrm{q}}(\nu)]$ appears in $[\Delta_{\mathrm{q}}(\lambda) \otimes \Delta_{\mathrm{q}}(\omega_1)^{\otimes s}]$ (for $0 \leq s < r$) have the property that $\nu_{k+m\ell}=0$, hence $\underline{\nu}(\oneh)=1$. Hence the $[\Delta_{\mathrm{q}}(\nu) \otimes \Delta_{\mathrm{q}}(\omega_1)]$ never contain the summand $[\Delta_{\mathrm{q}}(\nu)]$. Thus $\OPB_{\overline{0}}^{[\oneh]}\underline{\nu}$ has to be used to get the tensor product decomposition by Theorem \ref{thm:actionanddecompB2}. After the shift this becomes the operator $\OPB_{\overline{0}}^{[z]}\underline{\nu}$.
\end{proof}

\subsection{Linkage and operators for even $\ell$}

The first difference in case $\ell$ even is that $\ell_\beta=2$ for the roots of the form $\beta=\beta_{ij}^\pm$. These were responsible in the odd case that all indices had to be taken modulo $\ell$ to obtain the operators and that one had to group $\re_p$ and $\rf_q$ for $p+q \in \ell\bZ$. Now this is replaced by the analogue statements with $\ellh$ everywhere. Thus we have to make the following assumption throughout the section
\[
\ell \text{ is even and } \ellh > 3.
\]
As before this is needed to avoid quantum symmetric pairs for small $\ellh$, the general arguments for linkage work the same.

Namely one obtains the two statements.

\begin{lemma} \label{lem:linkedEorFtypeBeven}
Let $\underline{\lambda} \in X^{+}_\rho$ and $\ell$ even. Assume that $\underline{\mu}=\re_r\underline{\lambda}$ and $\underline{\nu}=\re_s\underline{\lambda}$ are defined for $r \neq s$. Then $\mu$ and $\nu$ are linked if and only if $r \in s + (\ellh)\bZ$.

Assume that $\underline{\mu^\prime}=\rf_r\underline{\lambda}$ and $\underline{\nu^\prime}=\rf_s\underline{\lambda}$ are defined for $r \neq s$. Then $\mu^\prime$ and $\nu^\prime$ are linked if and only if $r \in s + (\ellh)\bZ$.
\end{lemma}
\begin{lemma} \label{lem:linkedEFtypedifferentBeven}
Let $\underline{\lambda} \in X^{+}_\rho$ and $\ell$ even. Assume that $\underline{\mu}=\re_r\underline{\lambda}$ and $\underline{\nu}=\rf_s\underline{\lambda}$ are defined for $r+1 \neq s$. Then $\mu$ and $\nu$ are linked if and only if $r \in -s + (\ellh)\bZ$.
\end{lemma}

The analogue of Lemma \ref{lem:linkedEFtypesameBodd} now depend on whether $\ellh$ is itself odd or even. This makes sense if one looks at the corresponding Dynkin diagrams in Section \ref{sec:quantumgroup} for $r=\ellh$. In case of $X_\rho^{1,+}$ the Dynkin diagram with $\ellh$ nodes always has the property that $\overline{\oneh}$ is $\Theta$-linked, but it depends on the parity of $\ellh$ whether $\overline{\nicefrac{\ell}{4}}$ is fixed in case $\ellh$ is even or $\overline{\nicefrac{\ell-2}{4}}$ is also $\Theta$-linked in case $\ellh$ is odd. Similarly for $X_\rho^{\oneh,+}$ the Dynkin diagram with $\ellh$ nodes always has $\overline{0}$ is fixed, but it depends on the parity of $\ellh$ whether the ``opposite'' side of the Dynkin diagram has another fixed label or a pair of $\Theta$-linked labels. Thus the analogue of Lemma \ref{lem:linkedEFtypedifferentBodd} has four possible combinations.

\begin{lemma} \label{lem:linkedEFtypesameBeven1}
Let $\ell$ even. 
\begin{enumerate}
\item Assume $\ellh$ odd, $\underline{\lambda} \in X^{1,+}_\rho$, and that $\underline{\mu}=\re_{r-1}\underline{\lambda}$ and $\underline{\nu}=\rf_{r}\underline{\lambda}$ are defined. Then $\mu$ and $\nu$ are linked if and only if $r \in \oneh + (\ellh)\bZ$.
\item Assume $\ellh$ even, $\underline{\lambda} \in X^{1,+}_\rho$, and that $\underline{\mu}=\re_{r-1}\underline{\lambda}$ and $\underline{\nu}=\rf_{r}\underline{\lambda}$ are defined. If $\mu$ and $\nu$ are linked then $r \in \oneh + (\ellh)\bZ$ or $r \in \nicefrac{\ell+2}{4} + (\ellh)\bZ$.
\item Assume $\ellh$ odd, $\underline{\lambda} \in X^{\oneh,+}_\rho$, and that $\underline{\mu}=\re_{r-1}\underline{\lambda}$ and $\underline{\nu}=\rf_{r}\underline{\lambda}$ are defined. If $\mu$ and $\nu$ are linked then $r \in \nicefrac{\ell+2}{4} + (\ellh)\bZ$.
\item Assume $\ellh$ even, $\underline{\lambda} \in X^{\oneh,+}_\rho$, and that $\underline{\mu}=\re_{r-1}\underline{\lambda}$ and $\underline{\nu}=\rf_{r}\underline{\lambda}$ are defined. Then $\mu$ and $\nu$ are not linked.
\end{enumerate}
\end{lemma}
\begin{proof}
We only sketch the proof here, since for each situation one has to go through the same procedure as in the proof of Lemma \ref{lem:linkedEFtypedifferentBodd}. 

For $\ellh$ odd and $\underline{\lambda} \in X^{1,+}_\rho$ the proof is nearly word for word the same, hence the result is also the same. 

For $\ellh$ even and $\underline{\lambda} \in X^{1,+}_\rho$ one obtains that $r \in \oneh + (\ellh)\bZ$ in case of the minimal gallery having length one and $\gamma_1=\beta_i$. But in the remaining cases that $\gamma_1$ is different one obtains either $r \in \oneh + (\ellh)\bZ$ if $m_1+m_p$ appearing in the proof of Lemma \ref{lem:linkedEFtypedifferentBodd} is even or $r \in \nicefrac{(\ell+2)}{4} + (\ellh)\bZ$ if $m_1+m_p$ is odd. In case that $r \in \nicefrac{(\ell+2)}{4} + (\ellh)\bZ$ one does not obtain the converse. The argument using a reflection at a hyperplane for $\beta_i$ does not work here, since $\ell_{\beta_i}=\ell$. 

In the case $\ellh$ odd and $\underline{\lambda} \in X^{\oneh,+}_\rho$ the case of a minimal gallery of length $1$ gives a contradiction, while the other cases give $r \in \nicefrac{\ell+2}{4} + (\ellh)\bZ$. Note that if $\ellh$ is odd, $\nicefrac{\ell+2}{4}$ is an integer, so the $r$ is well-defined.

Finally, the case $\ellh$ even and $\underline{\lambda} \in X^{\oneh,+}_\rho$ with the assumption of linked weights always gives a contradiction.
\end{proof}

Thus we are left with comparing $\underline{\lambda}$ and the image under one of the operators $\re_i$ or $\rf_i$. In the case of $\underline{\lambda} \in X^{1,+}_\rho$ the situation becomes quite easy.

\begin{lemma} \label{lem:linkedEIdtypeBeven1}
Let $\underline{\lambda} \in X^{1,+}_\rho$ and $\ell$ is even. Assume that $\underline{\mu}=\re_r\underline{\lambda}$ is defined. Then $\mu$ and $\lambda$ are not linked.
Assume that $\underline{\nu}=\rf_r\underline{\lambda}$ is defined. Then $\nu$ and $\lambda$ are not linked.
\end{lemma}
\begin{proof}
Fix $i$ via $\underline{\mu}=\underline{\lambda}-\varepsilon_i$ and $r+\oneh = \underline{\lambda}_i$. Then as in the proof of Lemma \ref{lem:linkedEIdtypeBodd1} one only has to consider the root $\beta_i$. And one assumes that the weights are linked with hyperplane $H_{\beta_i,m}$ strictly between the alcoves having to exist. Hence one obtains
\[
(\underline{\lambda},\beta_i^\vee)=2\underline{\lambda}_i=1+m\ell
\]
for some $m\in\bZ$. But $2\underline{\lambda}_i$ is an even integer, while $1+m\ell$ is always odd, hence this is a contradiction.

Again the second case follows since if $\underline{\nu}=\rf_r\underline{\lambda}$ is defined, then $\re_r\underline{\nu}=\underline{\lambda}$ holds.
\end{proof}

In case of $\underline{\lambda} \in X^{\oneh,+}_\rho$, the situation resembles the case of $\ell$ odd.

\begin{lemma} \label{lem:linkedEIdtypeBeven2}
Let $\underline{\lambda} \in X^{\oneh,+}_\rho$ and $\ell$ is even. Assume that $\underline{\mu}=\re_r\underline{\lambda}$ is defined. Then $\mu$ and $\lambda$ are linked if and only if $r \in (\ellh)\bZ$.

Assume that $\underline{\nu}=\rf_r\underline{\lambda}$ is defined. Then $\nu$ and $\lambda$ are linked if and only if $r \in (\ellh)\bZ$.
\end{lemma}
\begin{proof}
We proceed like in the proof of Lemma \ref{lem:linkedEIdtypeBeven1} and obtain the same equation
\[
(\underline{\lambda},\beta_i^\vee)=2\underline{\lambda}_i=1+m\ell
\]
for some $m\in\bZ$. But now $2\underline{\lambda}_i$ is an odd integer, hence $\underline{\lambda} = \oneh + m\ellh$ which then implies $r\in(\ellh)\bZ$.

Conversely if $r\in(\ellh)\bZ$, then $(\underline{\lambda},\beta_i^\vee) = 2\underline{\lambda}_i=1+m\ell$ for some $m\in\bZ$, hence $s_{\beta_i,m}(\underline{\lambda}=\underline{\mu}$ and $\lambda$ and $\mu$ are linked.

As before the case for $\rf_r$ is done by using that $\re_r\underline{\nu}=\underline{\lambda}$ under the assumptions.
\end{proof}

\subsection{Quantum symmetric pair action for even $\ell$} \label{sec:quantumsympairactionBeven}
Throughout the section we make the assumption that
\[
\ell \text{ is even and } \ellh > 3.
\]
In this situation the action of the quantum symmetric pair depends on two different choices. First whether $\ellh$ is even or odd and second whether one is looking at weights in $X_\rho^{1,+}$ or $X_\rho^{\oneh,+}$. We just list the definitions of the linear operators in the different situations. Note that we obtain all possible quantum symmetric pairs for the affine type $A$ Dynkin diagram that we described before, i.e. we can have either two fixed labels, two $\Theta$-linked label pairs, or one of each and in that case we will see that they differ in the sense that one of them is a standard quantum symmetric pair, the other one is a non-standard quantum symmetric pair.

We are considering now $\bZ/(\ellh)\bZ$ and $\bH/(\ellh)\bZ$ with the usual automorphism $\Theta$.

\begin{definition} \label{def:quantumsymoperatorsB2}
Assume $\ell$ even. For $\underline{a} \in \mathcal{S}_{\bZ}$ let $\overline{p} \in \bH/(\ellh)\bZ$ and for $\underline{a} \in \mathcal{S}_{\bH}$ let $\overline{p} \in \bZ/(\ellh)\bZ$. Then define for $\overline{p}\neq\Theta(\overline{p})$
\begin{align*}
\OPB_{\overline{p}}\, \underline{a} &= v^{T_{\ellh}^{e-f}(-\overline{p},\underline{a})}\sum_{j \in \overline{p}} v^{R_{\ellh}^{e-f}(j,\underline{a})} \re_j\,\underline{a} + \sum_{j \in -\overline{p}} v^{L_{\ellh}^{f-e}(j,\underline{a})} \rf_j\,\underline{a}.\\
\OPL_{\overline{p}}\, \underline{a} &= v^{T_{\ellh}^{f-e}(\overline{p},\underline{a})}v^{T_{\ellh}^{e-f}(-\overline{p},\underline{a})} \underline{a}.
\end{align*}
For $\overline{p}=\Theta(\overline{p})$, but $0 \notin \overline{p}$ define
\[
\OPB_{\overline{p}} \underline{a} = v^{-1} v^{T_{\ellh}^{e-f}(\overline{p},\underline{a})}\sum_{j \in \overline{p}} v^{R_{\ellh}^{e-f}(j,\underline{a})} \re_j\underline{a} + \sum_{j \in \overline{p}} v^{L_{\ellh}^{f-e}(j,\underline{a})} \rf_j\underline{a},
\]
while for $0 \in \overline{i}$ $z \in \oneh + (\ellh)\bZ$ define
\begin{align*}
\OPB_{\overline{p}} \underline{a} &= v^{-1} v^{T_{\ellh}^{e-f}(\overline{p},\underline{a})}\sum_{j \in \overline{p}} v^{R_{\ellh}^{e-f}(j,\underline{a})} \re_j\underline{a} + \sum_{j \in \overline{p}} v^{L_{\ellh}^{f-e}(j,\underline{a})} \rf_j\underline{a} + v^{T_{\ellh}^{e-f}(\overline{p},\underline{a})}\underline{a}, \text{ and}\\
\OPB_{\overline{0}}^{[z]} \underline{a} &= v^{-1} v^{T_{\ellh}^{e-f}(\overline{0},\underline{a})}\sum_{j \in \overline{0}} v^{R_{\ellh}^{e-f}(j,\underline{a})} \re_j\underline{a} + \sum_{j \in \overline{0}} v^{L_{\ellh}^{f-e}(j,\underline{a})} \rf_j\underline{a} + \delta_{\underline{a}(z),0} v^{T_{\ellh}^{e-f}(\overline{0},\underline{a})}\underline{a}.
\end{align*}
This defines linear operators on $\cF$ and $\cF^{\oneh}$ that restrict to $\cF_N$ and $\cF_N^{\oneh}$ for any $N$.
\end{definition}

As one can see, in case that $\underline{a}\in \cS_{\bZ}$ and $\ellh$ is even there is no fixed index $\overline{p}\in\bH/(\ellh)\bZ$, instead there are two pairs of indices that are $\Theta$-linked, namely $\pm \overline{\oneh}$ and $\overline{\nicefrac{\ell \pm 2}{4}}$. 

In case that $\underline{a}\in \cS_{\bZ}$ and $\ellh$ is odd, the pair $\pm \overline{\oneh}$ is still $\Theta$-linked, but $\overline{\nicefrac{\ell}{4}}$ is a fixed index. But the quantum symmetric pair is a standard quantum symmetric pair, like in type $C$.

In case that $\underline{a}\in \cS_{\bH}$ and $\ellh$ is even, both $\overline{0}$ and $\overline{\nicefrac{\ell}{4}}$ are fixed, but the generator for $\overline{0}$ makes this a non-standard quantum symmetric pair like in type $B$ for odd $\ell$.

In case that $\underline{a}\in \cS_{\bH}$ and $\ellh$ is odd, we again have a non-standard operator for the fixed index $\overline{0}$ and a pair of $\Theta$-linked indices with $\overline{\nicefrac{\ell}{4}}$.

We skip the analogue of Lemma \ref{lem:comparetypeAB} in this case, it is just a combination of Lemmas \ref{lem:comparetypeA} and \ref{lem:comparetypeAB} depending on whether a generator for a fixed index is defined as in type $C$ or as in type $B$. Thus we get the corresponding actions

\begin{proposition} \label{prop:operator_relations_B2}
Assume $\ell$ even and $\rI \in \{\bZ/(\ellh)\bZ,\bH/(\ellh)\bZ\}$. The linear operators $\{\OPB_{\overline{p}} \mid \overline{p} \in \rI\}$ and $\{\OPL_{\overline{p}} \mid \overline{p} \in \rI, \Theta(\overline{p})\neq\overline{p}\}$ satisfy the relations of $B_v(\rI,\Theta)$ given in Definition \ref{def:quantumsymmpair} by substituting the operators for the generators with the same name.
\end{proposition}

Depending on the definition of the operators this follows either from the situation in type $C$ or from the case $\ell$ odd in type $B$. Then the corresponding theorems for the action and the decomposition of the tensor product can then be immediately formulated as follows.

\begin{theorem} \label{thm:actionanddecompB3}
Assume $\ell$ even. There exists an action of $B_v(\bH/(\ellh)\bZ,\Theta)$ on $\cF_N$ such that for $\lambda \in X^{1,+}$ the decomposition of $[\Delta_{\mathrm{q}}(\lambda) \otimes \Delta_{\mathrm{q}}(\omega_1)]$ in $[\Uqmod]$ with respect to the classes of Weyl modules is obtained from
\[
\sum_{\overline{p}\in \bH/(\ellh)\bZ} \OPB_{\overline{p}}\underline{\lambda} + \underline{\lambda}
\]
by projecting onto the subspace $\cF(C_N)$ and evaluating $v=1$. 

Furthermore if $[\Delta_{\mathrm{q}}(\mu)]$ and $[\Delta_{\mathrm{q}}(\nu)]$ appear in the decomposition for $\mu$ and $\nu$ linked, then there exists a unique $\overline{p}\in \bH/\ell\bZ$ such that $\underline{\mu}$ and $\underline{\nu}$ appear in $\OPB_{\overline{p}}\underline{\lambda}$.
\end{theorem}
\begin{proof}
This is the exact analogue of Theorem \ref{thm:actionanddecompB1}. But in case of $\ellh$ even there is no operator that produces a multiple of $\underline{\lambda}$ and in case of $\ellh$ odd the linear operator for the fixed index is standard, i.e. also does not produce a multiple of $\underline{\lambda}$. But since $\lambda \in X^{1,+}$, the tensor product always includes the class of $[\Delta_{\mathrm{q}}(\lambda)]$ in the Grothendieck group, hence one has to naively add it.
\end{proof}

And the analogue of Theorem \ref{thm:actionanddecompB2} is the following.

\begin{theorem} \label{thm:actionanddecompB4}
Assume $\ell$ even. There exists an action of $B_v(\bZ/(\ellh)\bZ,\Theta)$ on $\cF_N^{\oneh}$ such that for $\lambda \in X^{\oneh,+}$ the decomposition of $[\Delta_{\mathrm{q}}(\lambda) \otimes \Delta_{\mathrm{q}}(\omega_1)]$ in $[\Uqmod]$ with respect to the classes of Weyl modules is obtained from
\[
\sum_{\overline{p}\in \bZ/(\ellh)\bZ} \OPB_{\overline{p}}\underline{\lambda}
\]
if $\lambda_N>0$ and from
\[
\sum_{\overline{p}\in \bZ/(\ellh)\bZ,\overline{p}\neq\overline{0}} \OPB_{\overline{p}}\underline{\lambda} + \OPB_{\overline{0}}^{[\oneh]}\underline{\lambda}
\]
if $\lambda_N=0$, by projecting onto the subspace $\cF(C_N)$ and evaluating $v=1$. 

If $[\Delta_{\mathrm{q}}(\mu)]$ and $[\Delta_{\mathrm{q}}(\nu)]$ appear in the decomposition for $\mu$ and $\nu$ linked, then there exists a unique $\overline{p}\in \bZ/\ell\bZ$ such that $\underline{\mu}$ and $\underline{\nu}$ appear in $\OPB_{\overline{p}}\underline{\lambda}$.
\end{theorem}

Note in case of $\ellh$ even and for $\lambda \in X^{\oneh,+}$, there exist the operator $\OPB_{\overline{\nicefrac{\ell}{4}}}$ that has a fixed index, but it is of the same form as in type $C$, hence does not produce a multiple $\underline{\lambda}$. Hence only $\OPB_{\overline{0}}\underline{\lambda}$ needs to be replaced by $\OPB_{\overline{0}}^{[\oneh]}\underline{\lambda}$.

We skip the discussion of an analogues of Proposition \ref{prop:operatorsasdecomp} and Proposition \ref{prop:operatorsasdecompBhint}. The construction is analogous to type $C$ for the space $\cF_k$ and analogous, including the same manipulations, as in type $B$ (with $\ell$ odd) for $\cF_k^{\oneh}$.

\section{Type D and beyond} \label{sec:typeDbeyond}

The type $D$ case is part of the second authors Master thesis and the details will be published separately. We give a rough summary of what happens in this situation.

In type $D_N$ ($N>3$), the choices of $\varepsilon_i$ and the invariant bilinear form are precisely the same as in type $B$. The positive roots can be chosen as 
\[
\Phi^+ = \{ \beta_{i,j}^\pm=\varepsilon_i\pm\varepsilon_j \mid 1 \leq i < j \leq N \}.
\]
Note that all roots are short and hence equal to their own coroots and $\ell_\beta=\ell$ for all positive roots. Integral weights have a similar structure as in type $B$ with
\begin{align*}
X^{1,+} &= \{ (\lambda_1,\ldots,\lambda_N) \mid \lambda_i \in \bZ, \lambda_i \geq \lambda_{i+1}, \text{ and } \lambda_{N-1} \geq |\lambda_N|\} \text{ and}\\
X^{\oneh,+} &= \{ (\lambda_1,\ldots,\lambda_N) \mid \lambda_i \in \bH, \lambda_i \geq \lambda_{i+1}, \text{ and } \lambda_{N-1} \geq |\lambda_N|\}.
\end{align*}
With $\rho = (N-1)\varepsilon_1+(N-2)\varepsilon_2 +\ldots + \varepsilon_{N-1} \in X^{\oneh,+}$ we thus get
\begin{align*}
X^{1,+}_\rho &= \{ (\underline{\lambda}_1,\ldots,\underline{\lambda}_N) \mid \underline{\lambda}_i \in \bH,\, \underline{\lambda}_i > \underline{\lambda}_{i+1}, \text{ and } \lambda_{N-1} \geq |\lambda_N|\},\\
X^{\oneh,+}_\rho &= \{ (\underline{\lambda}_1,\ldots,\underline{\lambda}_N) \mid \underline{\lambda}_i \in \bZ,\, \underline{\lambda}_i > \underline{\lambda}_{i+1}, \text{ and } \lambda_{N-1} \geq |\lambda_N|\}.
\end{align*}
In analogy to type $B$, the combinatorial Fock space of type $D$ decomposes into two summands $\cF^{\oneh}(D_N)$ and $\cF^{1}(D_N)$. Furthermore the tensor product decomposition in type $D$ is the same as in type $C$ in Proposition \ref{prop:tensorproductC}, especially there is no complication with a $\Delta_v(\lambda)$ summand as in type $B$.

\begin{remark} \label{rem:signedsquences}
The condition for $\lambda$ to be dominant does not depend on $\lambda_N$ itself, but rather on its absolute value. Thus $(\lambda_1,\ldots,\lambda_{N-1},\lambda_N)+\varepsilon_N$ is dominant if and only if $(\lambda_1,\ldots,\lambda_{N-1},-\lambda_N)-\varepsilon_N$ is dominant. In terms of sequences this can be interpreted as saying that a $1$ should be placed at position $|\underline{\lambda}_N|$ and the sign of $\lambda_N$ needs to be recorded as well. 
\end{remark}

In case of $X^{\oneh,+}_\rho$, following Remark \ref{rem:signedsquences} one sees that one can embed $\cF^{\oneh}(D_N)$ into $\cF_N^{\oneh,+}\oplus \cF_N^{\oneh,-}$, where both 
spaces are isomorphic to $\cF_N^{\oneh}$. For this let $\underline{\lambda}^\prime$ be equal to $\underline{\lambda}$ except $\underline{\lambda}^\prime_N=|\underline{\lambda}_N|$. Then $\underline{\lambda}$ is mapped to $(\underline{a}_{\underline{\lambda}},0)$ if $\underline{\lambda}_N>0$ and to $(0,\underline{a}_{\underline{\lambda}^\prime})$ otherwise.
The moving operators from Definition \ref{def:movingoperators} are defined on $\cF_N^{\oneh,+}\oplus \cF_N^{\oneh,-}$, since they are defined on $\cF_N^{\oneh}$.

The interplay between moving operators and linkage, is analogous to types $C$. Hence the operators $\OPB_{\overline{p}}$ and $\OPL_{\overline{p}}$ are given as in Definition \ref{def:quantumsymoperatorsC}, up to changing the index set from $\bH/\ell\bZ$ to $\bZ/\ell\bZ$, and give an action of $B_v(\bZ/\ell\bZ,\Theta)$.

In addition, one introduces an operator $\OPB$ that goes between $\cF_N^{\oneh,+}$ and $\cF_N^{\oneh,-}$, since $(\lambda_1,\ldots,\lambda_{N-1},\oneh)-\varepsilon_N=(\lambda_1,\ldots,\lambda_{N-1},-\oneh)$, but both are represented by the same sequence, but in different components. One defines $\OPB(\underline{a},0)=(0,\underline{a})$ if $\underline{a}(\oneh)=1$ and the operator is zero otherwise, and similarly $\OPB(0,\underline{a})=(\underline{a},0)$ if $\underline{a}(\oneh)=1$ and zero otherwise. It follows that only $\OPB$ and $\OPB_{\overline{0}}$ can produce linked weights, hence one needs to consider $\OPB_{\overline{0}} + \OPB$.

The analogue of Theorem \ref{thm:actionanddecompB2} holds, with the difference that $\OPB_{\underline{0}}$ is replaced by $\OPB_{\overline{0}} + \OPB$. The limit construction from Proposition \ref{prop:operatorsasdecomp} does not work in type $D$ as one cannot regard a dominant weight $\lambda$ with $\lambda_N<0$ as a dominant weight for a bigger rank. If one restricts to this case then one essentially recovers the combinatorics of type $C$. 

The case of $X^{1,+}_\rho$ brings additional complications with it, since $\lambda_N=0$ does not allow for a ``good'' embedding into a sum of Fock spaces $\cF^{1,+} \oplus \cF^{1,-}$, where these spaces are defined as in the half-integer case. One instead embeds $\underline{\lambda}$ with $\lambda_N=0$ diagonally. This allows for an analogue of Theorem \ref{thm:actionanddecompC}, but depending on the starting weight, $(\underline{\lambda},0)$, $(0,\underline{\lambda})$ or $(\underline{\lambda},\underline{\lambda})$ represents the class of the Weyl module $[\Delta_{\mathrm{q}}(\lambda)]$ and similar for the occurring Weyl modules in the decompositions. The limit construction does not work for the same reason as above.

\begin{remark}
As a concluding remark in type $D$ one can note that the combinatorial calculations are easier than in type $B$ and $C$, since it is of simply-laced type. But the action of the quantum symmetric pair is further away from the combinatorics of the tensor product multiplicities and the limit construction is not possible.

This kind of complication in type $D$ has some analogy in \cite{ES}. Here the type $D$ situation is investigated in terms of category $\cO$ and certain generators of the quantum symmetric pair do not act as indecomposable translation functors, while in the analogous situation in type $B$ they would.
\end{remark}

We focused in this paper on non-exceptional Lie algebras. The general ideas can of course be transferred to exceptional cases as well. There is in general no clear analogue of the representation $\Delta_{\mathrm{q}}(\omega_1)$, i.e. a ``natural'' representation. For example, in type $G_2$ one can use the specialization and quantization of the natural representation of $\mathfrak{so}_8(\bC)$, since $\mathfrak{so}_8(\bC)$ contains the Lie algebra of type $G_2$ as a fixed point Lie algebra. But there is nothing analogous for the other exceptional types.

As for the analogue of the quantum symmetric pair, it is reasonable to expect that in type $G_2$ the acting algebra has some relationship to the quantum symmetric pair and an automorphism of order $3$. This is in analogy to the fact that the quantum symmetric pair is constructed from the quantum affine algebra using an automorphism of order $2$. It is a priori not clear what the corresponding object in other exceptional cases would look like or come from.




\end{document}